\documentclass{amsart}
\usepackage{lipsum}
\usepackage[dvips]{graphicx}
\usepackage{amssymb,latexsym}
\usepackage[ansinew]{inputenc}
\usepackage{mathrsfs}
\usepackage{psfrag}
\usepackage[all]{xy}
\usepackage{color}
\usepackage{graphicx,caption}
\usepackage{subfig}
\usepackage{hyperref}

\usepackage{comment}
\usepackage{amsfonts}
\usepackage{float}
\usepackage{amsmath,amssymb}
\usepackage{lscape}
\usepackage{tikz}
\usepackage{tikz-3dplot} 
\usetikzlibrary{arrows.meta,decorations.pathmorphing,patterns}

\newtheorem{corollary}{Corollary}[section]
\newtheorem{proposition}{Proposition}[section]
\newtheorem{theorem}{Theorem}

\newtheorem{definition}{Definition}[section]
\newtheorem{remark}{Remark}[section]

\newcommand{\R}{\mathbb{R}}

\def\qed{\hfill $\Box$}

\def\t{\noindent}

\def\b0{\mbox{\boldmath $0$}}

\begin{document}

\title{ON THE CONTACT OF SURFACES IN $4$-SPACE WITH 2-PLANES AND THEIR APPARENT CONTOURS}
\author{JORGE LUIZ DEOLINDO-SILVA and MOSTAFA SALARINOGHABI}



\begin{abstract} 
We investigate the local geometry of generic smooth surfaces in $\mathbb R^4$ via the contact with $2$-planes and the associated apparent contour.  We study the $\mathcal A$-singularities of parallel projections of such surfaces along planes to transverse planes. When the projection plane does not contain an asymptotic direction, we show that, at hyperbolic or elliptic (respectively, parabolic) points, there exist up to ten (respectively, seven) tangent directions determining planes along which the projection exhibits $\mathcal A$-singularities of type butterfly or worse. Moreover, we prove that the locus of points where the discriminant of the equation defining these directions vanishes generically forms regular curves in the hyperbolic and elliptic regions, and isolated points in the parabolic set of $M$.
When the projection plane contains an asymptotic direction, we establish connections between the singularities of parallel projections, orthogonal projections to hyperplanes, and height functions. We further study the apparent contour associated with the parallel projection. When the projection is a fold, we prove a Koenderink-type theorem relating the Lipschitz--Killing curvature of the surface to the curvature of the apparent contour. Moreover, we show that elliptic or inflections points of the surface give rise to vertices of the apparent contour, whereas hyperbolic and parabolic points give rise to inflections. These phenomena are then characterized in terms of the singularities of the projection. In the non-fold case, we show that the apparent contour admits singularities of type $(t^k,t^\ell)$-cusp associated with the singularities of the projection.
\end{abstract}

\maketitle

\section{Introduction}
 
Singularity theory has proven to be a powerful framework for studying smooth surfaces in 4-space. The contact of a surface in $\mathbb{R}^4$ with flat objects (such as hyperplanes, planes, and lines) reveals important aspects of its geometry (see the recent book \cite[Chapter 7]{IFRT-book} for references). For instance, while the contact with lines and hyperplanes yields robust information at hyperbolic and parabolic points, the contact with $2$-planes can provide characterizations both at hyperbolic and parabolic points, as well as at elliptic points. More recently, the contact of special classes of surfaces in $\mathbb{R}^4$ (such as Riemann and ruled surfaces) with 2-planes has been investigated within the affine structure of $\mathbb{R}^4$, and this approach has proven to be a valuable tool for extracting geometric information, including at elliptic points (see \cite{DT, DeolindoRuled}). However, these perspectives have been only partially explored in the context of generic surfaces in $\mathbb{R}^4$. Our aim in this paper is to investigate the local geometry of generic smooth surfaces in $\mathbb{R}^4$ associated with contact with $2$-planes, within the affine structure of $\mathbb R^4$.

The contact of surfaces with $2$-planes was first investigated in \cite{Nogueira}, with a particular focus on the relationship between the bifurcation sets of projections onto planes and height functions. However, that approach did not lead to significant geometric consequences for the surface. The contact between a surface $M$ and $2$-planes is captured by the singularities of parallel projections $P_\pi$ of $M$ along planes $\pi$ onto transverse planes.

In this paper, we characterize the $\mathcal A$-singularities of the parallel projections $P_\pi$ and describe robust geometric features of the surface $M$. We show that, when $\pi$ does not contain an asymptotic direction, there exist up to ten tangent directions at elliptic or hyperbolic points (respectively, up to seven tangent directions at parabolic points) determining planes along which the projection exhibits butterfly-type $\mathcal A$-singularities or worse. Moreover, we prove that the locus of points where the discriminant of the equation defining such directions vanishes generically forms regular curves at elliptic and hyperbolic points, and isolated special points at parabolic points.

When $\pi$ contains an asymptotic direction ${\bf u}$, we describe geometric information about $M$ which is related both to the singularities of the orthogonal projections to hyperplanes $P_{{\bf u}}$ along ${\bf u}$ and to the singularities of the height function $h_{\bf w}$ in the binormal direction ${\bf w}$ associated with ${\bf u}$. For instance, we show that a singularity of Type $11_5$ of $P_\pi$ occurs generically along a curve on $M$, called the {\it $11_5$-curve}, which coincides with the locus of singularities of Type $B_2$ of the orthogonal projection $P_{\bf u}$ (the {\it $B_2$-curve}). Furthermore, singularities of Type $4_4$ of $P_\pi$ occur generically along a curve on $M$, called the {\it $4_4$-curve}, which intersects the $S_2$-curve transversally at an $S_3$-singularity of $P_{\bf u}$. Moreover, singularities of Type $18$ and $19$ of $P_\pi$ occur at isolated points on this curve and correspond to the $A_4$-singularities of $h_{\bf w}$ and the $C_3$-singularities of $P_{\bf u}$, respectively.

Finally, we define the apparent contour of the surface $M$ along the plane $\pi$ as the image of the singular set $\Sigma_\pi$ of the projection under the map $P_\pi$. When the singular set is a regular curve, the projection $P_\pi$ may be $\mathcal A$-equivalent either to a fold or to the germ $(x,xy+y^k)$ $k\geq3$. In the first case, we prove a Koenderink-type theorem establishing that the Lipschitz--Killing curvature of $M$ coincides with the product of the curvature of the apparent contour and the curvature of the normal section along a non-asymptotic direction. We further show that, at elliptic or inflection points of $M$, the apparent contour exhibits vertices, whereas at hyperbolic or parabolic points of $M$ it exhibits inflections. 
In particular, at elliptic points, vertices of order $4$ occur generically along regular curves on the surface $M$, while vertices of order $5$ appear at isolated points on these curves. On the other hand, if $p$ is a hyperbolic point, inflections of order $2$ occur along the $11_5$-curve, whereas inflections of order $3$ arise at points where the projection $P_\pi$ has a singularity of Type $18$ on this curve. In the second case, we prove that the apparent contour admits singularities of type $(t^k,t^\ell)$-cusp, directly associated with the singularities of type $(x,xy+y^k)$ $k\geq3$ of the projection.

The paper is organized as follows. In Section \ref{section2}, we recall some results on the flat geometry of surfaces in $\mathbb R^4$. In Section \ref{sec:singprojc}, we define parallel projections and prove geometric characterizations of the surface $M$ in terms of the $\mathcal A$-singularities of the parallel projection $P_\pi$. Finally, in Section \ref{Apparent contour}, we introduce the apparent contour and study its inflections, vertices, and singularities in relation to the singularities of $P_\pi$.

\section{Preliminary}\label{section2}

We give a brief review and establish some notation concerning a smooth surface $M$ in $\mathbb R^4$. The differential geometry of generic surfaces in $\mathbb{R}^4$ has been extensively studied (see for example \cite{bruce-nogueira,BruceTari,DK,Ronaldoetal,little,dmm,projsrfR4,OsetTari}). Let $M$ be a regular surface in the Euclidean space $\mathbb R^4$. Consider a point $p\in M$ and a unit circle in $T_pM$ parametrized by $\theta \in [0,2\pi]$. The \emph{curvature ellipse} $\eta(\theta)$ in the normal plane $N_pM$, is the image of this unit circle described by a pair of quadratic forms $(Q_1, Q_2)$ (\cite{little}). Points on the surface are classified according to the position of the point $p$ with
respect to the ellipse ($N_pM$ is viewed as an affine plane through $p$). The point $p$ is
called \emph{elliptic/parabolic/hyperbolic} if it is inside/on/outside the ellipse at $p$, respectively.

The pair of quadratic forms is the $2$-jet of the
1-flat map $\mathcal F:{({\mathbb R}}^2,0)\to {{(\mathbb R}}^2,0)$ (i.e. without constant or linear terms) whose graph,
in orthogonal coordinates, is locally the surface $M$. Using a different approach to the geometry of surfaces in ${{\mathbb R}}^4$ given in \cite{bruce-nogueira}, each point on the surface determines a pair of quadratics:
$$ 
(Q_1,Q_2)=(ax^2+2bxy+cy^2,lx^2+2mxy+ny^2). 
$$

Representing a binary form $Ax^2+2Bxy+Cy^2$ by its coefficients $(A:B:C)\in {\mathbb R}P^2$, there
is a cone $\Gamma: B^2-AC=0$ representing the perfect squares.
If the forms $Q_1$ and $Q_2$ are linearly
independent, then they determine a line in the projective plane ${\mathbb R} P^2$ and the cone is a
conic. This line meets the conic $\Gamma$ in 0, 1 or 2 points according as
$\delta(p)<0,=0,>0$ where \[ \delta(p)=(an-cl)^2-4(am-bl)(bn-cm).\]
Thus a point $p$ is {elliptic/parabolic/hyperbolic} if $\delta(p)<0/=0/>0$. The set of points in $M$
 where $\delta=0$ is called the \emph{parabolic set} of $M$ and is denoted by $\Delta$.
If $Q_1$ and $Q_2$ are dependent at a point $p$, then the point is called \emph{inflection point}.


Consider the action of the group ${\mathcal G} = GL(2, \mathbb{R}) \times GL(2, \mathbb{R})$ on pairs of binary forms. The classification of these ${\mathcal G}$-orbits is presented in Table \ref{Gibson} (see, e.g., \cite{gibson} for further details):

\begin{table}[h]
\caption{$GL(2,{\mathbb R})\times GL(2,{\mathbb R})$-classes of pairs of binary forms. } \label{Gibson}
\begin{tabular}{cc}
\hline
Normal form & Name\\
\hline
$(x^2,y^2)$ & \mbox{hyperbolic point} \\
$(xy,x^2-y^2)$ & \mbox{elliptic point}\\
$(xy,y^2)$ &\mbox{parabolic point}\\
$(x^2+ y^2,0)$ or $(xy,0)$ &\mbox{inflection point}\\
$(x^2,0)$ &\mbox{degenerate  inflection point}\\
$(0,0)$ &\mbox{degenerate inflection point}\\
\hline
\end{tabular}

\end{table}

The geometrical characterization of points on $M$ using singularity theory is carried out in \cite{dmm} via the family of height functions
$h : M\times S^3 \to {{\mathbb R}} $ given by $h(p,v)=\langle p,v\rangle$,
 where $S^3$ denotes the unit sphere in ${\mathbb R}^4$. For $v$ fixed, the height function $h_v(p)=h(p,v)$ is singular if and only
if $v\in N_pM$. It is shown in \cite{dmm} that elliptic points are non-degenerate critical points of
$h_v$ for any $v\in N_pM$. At a hyperbolic point, there are exactly two directions in $N_pM$, labelled \emph{binormal directions},
such that $p$ is a degenerate critical point of the corresponding height functions. The two binormal directions coincide at a parabolic point.
The binormal directions ${\bf w}=(\lambda_1,\lambda_2)\in N_pM$ can be determined by the zeros of 
\begin{equation}\label{binormaleq}
\mathcal B({\bf w})=(ac-b^2)\lambda_1^2 +(an+cl-2bm)\lambda_1\lambda_2+(ln-m^2)\lambda_2^2,
\end{equation}

The direction of the kernel of the Hessian of the height functions along a binormal
direction is an \emph{asymptotic direction} associated to the given binormal direction (\cite{dmm}).
The asymptotic directions are labelled by conjugate directions in \cite{little}, and are defined as the directions along $\theta$ such that the curvature vector $\eta(\theta)$ is tangent to the curvature
ellipse (see also (\cite{Ronaldoetal, dmm})). This is invariant under the action of $GL(2,\mathbb R)$ in the normal plane, so the concept of asymptotic direction is affine invariant. 
The asymptotic directions ${\bf u}=(dx,dy)\in T_pM$ can be determined by the zeros of 
\begin{equation}\label{asymp}
\Omega({\bf u})=(am-bl)dx^2 +(an-cl)dxdy+(bn-cm)dy^2,
\end{equation}
and the discriminant of the $\Omega=0$ is given by the zero set of the function $\delta(p)$ and coincides with the parabolic set of the surface (\cite{BruceTari}). 
So if $p$ is not an inflection point, then there are $2/1/0$ asymptotic directions at $p$ depending on $p$ being a hyperbolic/parabolic/elliptic point. If $p$ is an inflection point, then every direction in
$T_pM$ is asymptotic. 

Asymptotic directions can also be described via the singularities
of the orthogonal projections of $M$ to hyperplanes (see \cite{bruce-nogueira}). The family of orthogonal projections is given by

$$
\begin{array}{rcl}
P : M\times S^3& \to& TS^3 \\
(p,{\bf u})&\mapsto&({\bf u},p-\langle p,{\bf u}\rangle {\bf u}).
\end{array}
$$
For ${\bf u}$ fixed, the projection can be viewed locally at a point $p\in M$ as a map germ $P_{\bf u}:{(\mathbb R}^2,0)\to
({\mathbb R}^3,0)$. If we allow smooth changes of coordinates in the source and target
(i.e., consider the action of the Mather group $\mathcal{A}$) then the generic $\mathcal A$-singularities of $P_{\bf u}$
are those that have $\mathcal A$-codimension less than or equal to 3 (which is the dimension of $S^3$). These are listed in Table \ref{mond} (see  \cite{mond}).
\begin{table}[htp]
\begin{center}
\begin{tabular}{c|c|c}
\hline
Name & Normal form & ${\mathcal A}_e$-codimension\cr
\hline
\hline
Immersion & $(x,y,0)$ &0\cr
Cross-cap & $(x,y^2,xy)$ &0\cr
$B^{\pm}_{k}$ & $(x, y^2, x^2y \pm y^{2k+1}), k=2,3$ & $k$ \cr
$S^{\pm}_{k}$ & $(x, y^2, y^3 \pm x^{k+1}y), k=1,2,3$ & $k$\cr
$C^{\pm}_{k}$ & $(x, y^2, xy^3 \pm x^{k}y), k=3$ & $k$\cr
$H_k$ & $(x,xy+y^{2k+2},y^3), k=2,3$ & $k$\cr
$P_3(c)$ & $(x, xy + y^3, xy^2 + cy^4),\;c\neq0,\frac{1}{2},1,\frac{3}{2}$ & $3^*$\cr
\hline
\end{tabular}

$(*)$: the codimension in brackets is that is of the $\mathcal A$-stratum.
    \caption{The generic local singularities of projections of $M$ to hyperplanes.}\label{mond}
    
\end{center}
\end{table}

The projection $P_{\bf u}$ is singular at $p$ if and only if ${\bf u}\in T_pM$. The singularity is a cross-cap unless ${\bf u}$ is an asymptotic direction at $p$. As $P$ has  3-parameters, the ${\mathcal A}_e$-codimension $2$ singularities occur generically on curves on the surface, and the ${\mathcal A}_e$-codimension $3$ ones occur at special points on these curves. The closure of the set of points $p$ on $M$ for which there exists a projection $P_{\bf u}$ having an $S_2$ (resp. $B_2$, $H_2$)-singularity at $p$ is called by $S_2$-curve (resp. $B_2$-curve and $H_2$-curve). The $H_2$-curve coincides with the parabolic set (\cite{bruce-nogueira}). The $B_2$-curve of $P_{\bf u}$, with ${\bf u}$ asymptotic, is also the $A_3$-set of the height functions along the binormal direction associated to ${\bf u}$ (\cite{bruce-nogueira}). This curve meets the parabolic set tangentially at a $P_3(c)$-singularity (it is called $P_3(c)$-point in \cite{Deolindo}) of the projection along the unique asymptotic direction (\cite{BruceTari}) and intersects
the $S_2$-curve transversally at a $C_3$-singularity (see Figure \ref{fig:SpecialC} for their configurations). At inflection points, the parabolic set has a Morse singularity. For more details on inflection points, see \cite{bruce-nogueira,Ronaldoetal}.

\begin{figure}[htb]
\includegraphics[ width=7.5cm, height=4cm]{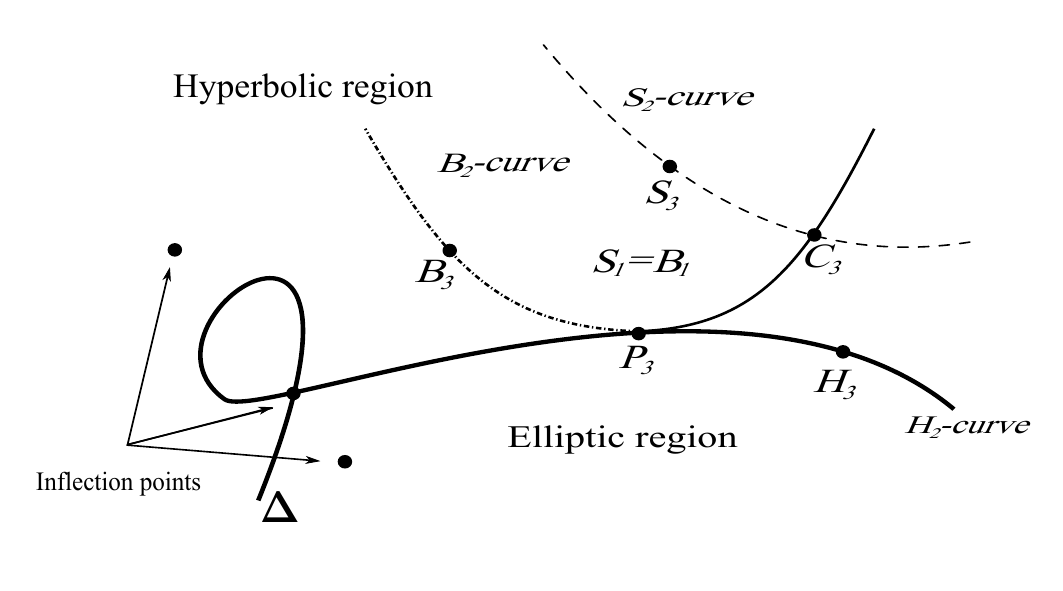}
\caption{Curves and special  points on $M$ (\cite{dmm,projsrfR4}).}
\label{fig:SpecialC}
\end{figure}

Let ${\bf E = \{e_1, e_2, e_3, e_4\}}$ be the standard basis of $\mathbb{R}^4$, and let $M$ be parametrised in Monge form by
\[
F(x,y)=(x,y,f^1(x,y),f^2(x,y)),
\]
where
\begin{equation}\label{Monge} \displaystyle \begin{array}{l} \displaystyle f^1(x,y)= \sum_{i+j\ge 2} {a}_{ij}x^i y^j, \;\; \displaystyle f^2(x,y)=\sum_{i+j\ge 2} {b}_{ij}x^i y^j. \end{array} \end{equation}
Then, at the origin $p=(0,0)$, the tangent space $T_pM$ is generated by ${\bf e}_1$ and ${\bf e}_2$. The coefficients of $(Q_1,Q_2)$ at $p$ are given by
\[
a=\tfrac12 f^1_{xx}, \quad b=\tfrac12 f^1_{xy}, \quad c=\tfrac12 f^1_{yy}, \quad
l=\tfrac12 f^2_{xx}, \quad m=\tfrac12 f^2_{xy}, \quad n=\tfrac12 f^2_{yy}.
\]
Thus, from Equation \eqref{asymp}, a vector ${\bf u}=(\alpha_1,\alpha_2,0,0)$ represents an asymptotic direction at $p$ if and only if
\begin{equation}\label{asymptoticlines}
\Omega({\bf u})=(a_{02}b_{11}-a_{11}b_{02})\,\alpha_2^2
+2(a_{02}b_{20}-a_{20}b_{02})\,\alpha_1\,\alpha_2
+(a_{11}b_{20}-a_{20}b_{11})\,\alpha_1^2
\end{equation}
vanishes. Moreover, the height function $h_{\bf w}$ is degenerate at $p$ in the normal direction ${\bf w}=(0,0,s,t)$ if and only if
\begin{equation}\label{binormallines}
\mathcal{B}({\bf w})=(4a_{02}a_{20}-a_{11}^2)s^2
+\big(4(a_{20}b_{02}+a_{02}b_{20})-2a_{11}b_{11}\big)st
+(4b_{02}b_{20}-b_{11}^2)t^2
\end{equation}
vanishes. In this case, ${\bf w}$ is a binormal direction at $p$.

\section{Parallel projections along planes}\label{sec:singprojc}

We are interested in the affine geometry of smooth surfaces using contact with 2-planes. The geometric characterization of points on a generic surface $M$ in 4-space obtained using pairs of quadratic forms can be recovered by considering the contact of the surface with lines and hyperplanes (see \cite{IFRT-book}). The geometry of elliptic points remains relatively unexplored from the perspective of singularity theory. In this context, the contact with $2$-planes is particularly relevant, as it not only characterizes the geometry of hyperbolic and parabolic points, but also reveals geometric features of elliptic points.

  
A plane $\pi^*$ is the kernel of a linear submersion $\psi:\mathbb R^4\to \mathbb \pi_t$, where $\pi_t$ is a transverse plane to $\pi^*$. Following Montaldi in \cite{Montaldi}, the contact of a surface $M$ with $\pi^*$ at a given point on $M$ is captured by the $\mathcal K$-singularities of the composite map 
$\psi\circ \varphi$, with $\varphi$ a local parametrisation of $M$. The map $\psi\circ \varphi$  is locally a map-germ $(\mathbb R^2,0)\to (\mathbb R^2,0)$. We consider here the action of the group $\mathcal A$ as it gives finer results and see Rieger's  $\mathcal A$-classification for corank 1 case in Table \ref{tab:rieger} and for corank 2 case in Table \ref{corank2}.

The set of all planes through the origin in $\mathbb{R}^4$ is the Grassmannian $Gr(2,4)$, a 4-dimensional manifold. To simplify calculations, Nogueira in \cite{Nogueira} introduced a scalar product on $\mathbb{R}^4$ to obtain a local parametrization of $Gr(2,4)$ and a simple expression for the family of projections along planes. However, his analysis focused only on certain contact planes associated with asymptotic or binormal directions. In contrast, our characterization is more general, regardless of whether the plane contains an asymptotic direction. Our approach is based on the same construction introduced in \cite{DT}.


Consider two transverse planes, $\pi^*$ and $\pi_t$, in $\mathbb{R}^4$. We choose a basis vector ${\mathcal F}=\{ {\bf f_1},{\bf f_2},{\bf f_3},{\bf f_4}\}$ of $\mathbb R^4$ such that ${\bf f_1}$ and ${\bf f_2}$ generate $\pi^*$, while ${\bf f_2}$ and ${\bf f_3}$ generate $\pi_t$. Let $\pi$ be a plane 
near to $\pi^*$ that can be generated by the two column vectors ${\bf u}$ and ${\bf v} $ of a matrix 
$$
\left(
\begin{array}{cc}
1&0\\
0&1\\
\alpha_1&\beta_1\\
\alpha_2&\beta_2
\end{array}
\right)
$$
where the coordinates of the vectors are with respect to the basis $\mathcal F$. 
Therefore, we can identify the set of planes near $\pi^*$ with an open 
neighbourhood $U$ of the origin in  $\mathbb R^4$. 
The plane $\pi$ above, being close to $\pi^*$, is still transverse to $\pi_t$, so we project along $\pi$ to $\pi_t$ (see \cite{DT}) for details). 

The parallel projection of  ${\bf x}=(x_1,x_2,x_3,x_4)_{\mathcal F}\in \mathbb R^4$ to $\pi_t$ is a point 
${\bf y}={\bf x}-x_1 {\bf u}-x_2 {\bf v}\in \pi_t$. Therefore, the expression for the family of parallel projections 
$
P: \mathbb R^4\times U \to \pi_t\equiv \mathbb R^2 
$  
along planes near $\pi^*$ to the plane $\pi_t$, 
is given by
$$
P({\bf x}, \pi)=(x_3-\alpha_1x_1-\beta_1x_2,x_4-\alpha_2x_1-\beta_2x_2),
$$
with $\pi$ generated by ${\bf u},{\bf v}$ above. If we choose another basis for $\pi_t$ or another transverse plane to $\pi^*$, then the expressions for the projections are equivalent under 
multiplication by an invertible $2\times 2$ matrix, so the resulting projections are $\mathcal A$-equivalent. 

Restricting the family $P$ to a surface $M\subset \mathbb R^4$ gives  the family of parallel projections of $M$ to transverse planes. We still denote this restriction by $P$ and by $P_{\pi}:M\to \mathbb R^2$, the map 
given by $P_{\pi}(p)=P(p,\pi)$, for $p\in M$. We also denote by  
$j_1^kP(p,\pi)$, the $k$-jet of $P_{\pi}$ at $p$.  By Montaldi's theorem in \cite{Montaldi2} if
$W$ is an $\mathcal A$-invariant variety in $J^k(M,\mathbb R^2)$, then, for a residual set of immersions of a surface $M$ in $\mathbb R^4$, 
the map $j_1^kP: M\times Gr(2,4)\to 
J^k(M,\mathbb R^2)$ 
is transverse to $W$.
This means that for a generic surface $M$, the only singularities that $P_{\pi}$ can have are those 
of {$\mathcal A_e$-codimension $\le 4=\dim(Gr(2,4))$} (of the stratum, in the presence of moduli). Furthermore, these singularities are $\mathcal A_e$-versally unfolded by the family $P$.

\begin{table}
\caption{$\mathcal A$-singularities of corank 1 map-germs with $d_e(g,\mathcal A)\le 4$ (\cite{rieger}).}
\begin{center}
\begin{tabular}{l l c}
\hline
Type & Normal form & $d_e(g,\mathcal A)$ \cr
\hline
$1$ (Immersion) & $(x,y)$ &$0$\cr
$2$ (Fold) & $(x,y^2)$ &$0$\cr
$3$ (Cusp) & $(x, xy + y^{3})$ & $0$ \cr
$4_k$ (Lips/beaks for $k=2$) & $(x,y^3 \pm x^{k}y), 2\leq k\leq5$ & $k-1$\cr
$5$ (Swallowtail) & $(x, xy+ y^4)$ & $1$\cr
$6$ (Butterfly) & $(x,xy+y^{5}\pm y^7)$ & $2$\cr
$7$ & $(x, xy+ y^5)$ & $3$\cr
$8$ & $(x,xy+y^{6}\pm y^8+\alpha y^9)$ & $4(3^*)$\cr
$9$ & $(x, xy+ y^6+y^9)$ & $4$\cr
$10$ & $(x,xy+y^{7}\pm y^9+\alpha y^{10}+\beta y^{11})$ & $6(4^*)$\cr
$11_{2k+1}$ & $(x,xy^2+y^4+y^{2k+1}), 2\leq k\leq4$ & $k$\cr
$12$ & $(x,xy^2+y^{5}+ y^6)$ & $3$\cr
$13$ & $(x, xy^2+ y^5\pm y^9)$ & $4$\cr
$15$ & $(x,xy^2+y^{6}+y^7\pm \alpha y^9)$ & $5(4^*)$\cr
$16$ & $(x, x^2y+ y^4\pm y^5)$ & $3$\cr
$17$ & $(x, x^2y+ y^4)$ & $4$\cr
$18$ & $(x,x^2y+xy^3+\alpha y^{5}+y^6\pm \beta y^7)$ & $6(4^*)$\cr
$19$ & $(x,x^3y+\alpha x^2y^2+ y^{4}+x^3y^2)$ & $5(4^*)$\cr
\hline
\end{tabular}

$(*)$: the codimension in brackets is that is of the $\mathcal A$-stratum.
\end{center}
 \label{tab:rieger}
\end{table}

\begin{table}[h]
\begin{center}
\caption{$\mathcal A$-singularities of corank 2 map-germs with $d_e(g,\mathcal A)\le 4$}\label{corank2}
\begin{tabular}{l l c}
\hline
Type & Normal form & $d_e(g,\mathcal A)$ \cr
\hline
$I_{2,2}^{l,m}$  & $(x^2+y^{2l+1},y^2+x^{2m+1})$, $1\leq m\leq l\leq3$ & $l+m\leq4$\cr
$II_{2,2}^{l}$  & $(x^2-y^2+x^{2l+1},xy)$, $1\leq l\leq2$ & $2l$\cr
\hline
\end{tabular}
\end{center}
\end{table}

 It is not hard to show that $P_{\pi}$ is singular at the origin if and only if $\dim(T_pM\cap \pi)\ge 1$. 
We consider the following two cases separately: $\dim(T_pM\cap \pi)=1$ and  $\pi=T_pM$. 

\subsection{Case $\dim(T_pM\cap \pi)=1$} \label{case1}

\

As we have fixed the basis $\bf E$ and $T_pM$, we change the earlier setting for the parallel 
projection and take a pair of generators ${\bf u},{\bf v}$ of $\pi\in Gr(2,4)$ in general form.
We take, without loss of generality 
\begin{equation}
{\bf u}=(1,\alpha,0,0)\in T_pM\cap \pi
\end{equation} 
and ${\bf v}=(0,\lambda,\mu,\beta)$, with 
$\mu^2+\beta^2\ne 0$. 
We consider the parallel projection along $\pi$ to the transverse plane $\pi_t$ generated by 
${\bf e_2}$ and ${\bf f}=(1,0,-\beta,\mu)$.  
Then, 
$$
P_{\pi}({\bf x}){\sim_{\mathcal A}}\left(x_2-\alpha x_1-\frac{\lambda}{\mu^2+\beta^2} (\mu x_3+\beta x_4),\, \beta x_3-\mu x_4\right),
$$
where we scaled the last coordinate by $\mu^2+\beta^2$, and where the coordinates of $P_{\pi}({\bf x})$ are with respect to the basis $\{{\bf e_2} ,{\bf f}\}$ of $\pi_t$. 
 
 The restriction of $P_{\pi}$ to the surface $M$ given by Monge form $F$ (still denoted by $P_{\pi}$) 
is the map-germ $P_{\pi}:(\mathbb R^2,0)\to (\mathbb R^2,0)$ given by
{\small
\begin{equation}\label{eq:case1}
P_{\pi}(x,y){\sim_{\mathcal A}}\left(y-\alpha x-\frac{\lambda}{\mu^2+\beta^2} (\mu f^1(x,y)+\beta f^2(x,y)),\, \beta f^1(x,y)-\mu f^2(x,y)\right).
\end{equation}
}
After one more change of coordinates, Equation (\ref{eq:case1}) takes the form
\begin{equation}\label{Proj1}
P_{\pi}(x,y)\sim_{\mathcal A} (y,g_{(\alpha,\lambda, \beta, \mu)}(x,y)).
\end{equation} 

It is clear that $P_{\pi}$ is a corank 1 map-germ.  Then we have the following observation about projection $P_\pi$.

\begin{proposition}[\cite{Nogueira, IFRT-book}] \label{prop:class}
With notation as above, let $M$ be a regular surface and
$\pi\in Gr(2,4)$ with $\dim(T_pM\cap\pi)=1$.
Denote by $\Sigma_\pi$ the singular set of the projection $P_\pi$.

\begin{enumerate}
\item[(i)]
$\Sigma_\pi$ is a regular curve if and only if
$j^2P_\pi$ is $\mathcal A^{(2)}$-equivalent to
$(y,x^2)$ or $(y,xy)$.

\item[(ii)]
Assume that $p$ is a non-elliptic and non-inflection point.
Then
$j^2P_\pi$ is $\mathcal A^{(2)}$-equivalent to $(y,0)$
if and only if the plane $\pi$ contains an asymptotic direction
${\bf u}$ and
\[
{\bf w}=(0,0,\beta,-\mu)
\]
is the binormal direction associated with ${\bf u}$.
In particular, $\Sigma_\pi$ exhibits a Morse
$A_1^\pm$-singularity or worse at $p$.

\item[(iii)]
If $p$ is an inflection point, then
$j^2P_\pi$ is $\mathcal A^{(2)}$-equivalent to $(y,0)$
whenever
$
{\bf w}=(0,0,\beta,-\mu)
$
is the unique binormal direction at $p$ corresponding to every asymptotic direction.
\end{enumerate}
\end{proposition}

\begin{proof} 
 Consider the projection $P_{\pi}$ given in \eqref{Proj1}. We shall write $g = g_{(\alpha,\lambda,\beta,\mu)}$, then the singular set $\Sigma_\pi$ of $P_{\pi}$  is the zero set of the function $g_x(x,y)=0$.
The result follows by considering the relevant jet of $g_x$ at the origin. We have 
{\small
\begin{eqnarray*}
\frac12j^1g_x(x,y)=\Big((\alpha b_{11}+2b_{20})\mu-(a_{11}\alpha+2a_{20})\beta\Big)x
+\Big((b_{11}+2\alpha b_{02})\mu-(2a_{02}\alpha+a_{11})\beta\Big)y.
\end{eqnarray*}
}
\noindent
{\normalfont(i)} $P_{\pi}$ is a fold (Type 2) if and only if  $(\alpha b_{11}+2b_{20})\mu-(a_{11}\alpha+2a_{20})\beta\neq 0$. Moreover $j^2P_\pi$ is $\mathcal A^{(2)}$-equivalent to $(y,xy)$ if and only if $(\alpha b_{11}+2b_{20})\mu-(a_{11}\alpha+2a_{20})\beta= 0$ and $(b_{11}+2\alpha b_{02})\mu-(2a_{02}\alpha+a_{11})\beta\neq0$. In both cases, $\Sigma_\pi$ is a regular curve at $p$.

\noindent
{\normalfont(ii)} Now $j^1g_x$ is identically zero if and only if 
$$(\alpha b_{11}+2b_{20})\mu-(a_{11}\alpha+2a_{20})\beta= (b_{11}+2\alpha b_{02})\mu-(2a_{02}\alpha+a_{11})\beta=0.$$
From Equations (\ref{asymptoticlines}) and (\ref{binormallines}) it is equivalent to $\Omega({\bf u})=0$ and $\mathcal B({\bf w})=0$, where ${\bf u}=(1,\alpha,0,0)\in T_pM\cap \pi$ and ${\bf w} = (0, 0, \beta,-\mu)\in N_pM$. 
If $p$ is not an inflection point, then these solutions characterize ${\bf u}$ as an asymptotic direction 
and ${\bf w}$ as the binormal direction associated to ${\bf u}.$
Since elliptic points admit neither asymptotic nor binormal directions, then at a non-elliptic point there exist planes $\pi$ containing asymptotic directions ${\bf u}$ 
for which $j^2P_{\pi}$ is $\mathcal{A}^{(2)}$-equivalent to $(y,0)$. In particular, $g_x$ has a Morse $A_1^\pm$-singularity  if and only if the discriminant of the quadratic terms of $g_x$ is nonzero.

\noindent
{\normalfont(iii)} At an inflection point, every tangent direction is asymptotic and the asymptotic-binormal correspondence degenerates. Therefore, whenever the projection plane $\pi$ contains an asymptotic direction ${\bf u}$ and a binormal direction ${\bf w}$, the conditions above imply that $j^1g_x$ vanishes identically. Hence $\Sigma_\pi$ is singular at $p$.

\qed
\end{proof}

\begin{remark}\label{rem:class}
{\normalfont
\begin{itemize}
\item[(i)] Note that, at elliptic points, $j^2P_{\pi}$ necessarily reduces, up to $\mathcal{A}^{(2)}$-equivalence, to either $(y, x^2)$ or $(y, xy)$. At parabolic, hyperbolic, or inflection points, we also obtain these singularities, together with those whose 2-jet is equivalent to $(y,0)$.
\item[(ii)] Furthermore, at non-elliptic points if ${\bf u}$ is an asymptotic direction, the projection $j^2P_{\pi}$ is $\mathcal{A}^{(2)}$-equivalence $(y,0)$ only when ${\bf w}$ is a binormal direction associated to ${\bf u}$. 
\end{itemize}
}
\end{remark}

According to Proposition \ref{prop:class} and Remark \ref{rem:class}, we must consider three possible $\mathcal{A}^{(2)}$-equivalence classes for the 2-jet of the projection $P_\pi$, namely $(y, x^2)$, $(y, xy)$, and $(y,0)$. 
 Therefore, to describe the local behavior of the projection near a generic point, we must analyze the elliptic, parabolic, hyperbolic, and inflection cases separately, taking into account the relation between asymptotic and binormal directions at the non-elliptic ones. We distinguish two cases in the analysis: ${\bf u}$ as a non-asymptotic direction in $\pi$, and ${\bf u}$ as an asymptotic direction.

\subsubsection{${\bf u}$ is a non-asymptotic direction in $\pi$}\label{ParabolicSec} 

In this case, we analyze the projection along $\pi$ when $p$ is an elliptic, parabolic, or hyperbolic point. Since ${\bf u}$ is not an asymptotic direction, the inflection point is not considered here. 
%

Observe that we can also set, without loss of generality, $\beta =1$ in expression (\ref{Proj1}) of $P_{\pi}$ (setting $\mu=1$ when $\beta=0$ does not give extra information). 
Then, the set $\Pi\subset Gr(2,4)$ of planes $\pi$ for which $P_{\pi}$ is singular at the origin 
can be parametrised by $(\alpha,\lambda,\mu)\in \mathbb R^3$. 

We have the following about the $\mathcal A$-singularities of $P_{\pi}$, where the type number refers to one in the first column of Table \ref{tab:rieger}. 
   
\begin{theorem}\label{theo:notInf}
Let $M$ and $\Pi$ be as above and $p$ an elliptic, hyperbolic or parabolic point. 
\begin{itemize}
\item[(a1)] For $\pi \in \Pi$ and away from a surface $\mathcal S \subset \Pi$, the singularity of the map-germ $P_{\pi}$ is of Type 2 {\rm (}fold{\rm)}. 

\item[(a2)]  For $\pi\in S$ and away from a curve $\mathcal C$ on $S$, the singularity of $P_{\pi}$ is of Type 3 {\rm (}cusp{\rm)}. 

\item[(a3)] At elliptic points, for $\pi \in \mathcal C$ except possibly for at most ten exceptional planes $\pi^e_i$, $i = 1, \ldots, 10$ the singularity of $P_{\pi}$ is of Type 5 (swallowtail).
At hyperbolic points, the same holds except possibly for ten planes $\pi^h_j$, $j = 1, \ldots, 10$, and at parabolic points except possibly for seven planes $\pi^p_k$, $k = 1, \ldots, 7$.

\item[(a4)] 
At elliptic points, the planes $\pi^e_i$ in {\rm (a3)} are determined by up to ten directions in $T_p M$ called the {\rm elliptic butterfly directions} and determined by Equation {\rm(\ref{ButterflyDirections})}. At hyperbolic points, the planes $\pi^h_j$ are also determined by up to ten directions called the {\rm hyperbolic butterfly directions} and determined by Equation {\rm(\ref{ButterflyDirections_hyp})}, and at parabolic points, the planes $\pi^p_k$ are determined by up to seven directions called the {\rm parabolic butterfly directions} and determined by Equation {\rm(\ref{ButterflyDirections_par})}:
\begin{equation}\label{ButterflyDirections}
{\tiny
\begin{array}{l}
(a_{03}b_{03}-b_{03}b_{12}+b_{04})\alpha^{10}+(a_{03}^2-b_{03}^2-2b_{03}b_{21}-b_{12}^2-b_{13}^2+a_{04})\alpha^9\\
+(a_{03}a_{12}-a_{03}b_{21}-a_{12}b_{12}-a_{21}b_{03}-b_{03}b_{12}-3b_{03}b_{30}
-3b_{12}b_{21}+b_{04}+b_{22})\alpha^8\\
+(2a_{03}^2-2a_{03}b_{30}-2a_{12}b_{21}
-2a_{21}b_{12}-2a_{30}b_{03}-2b_{03}^2-4b_{12}b_{30}-2b_{21}^2+2a_{04}\\+a_{22}+b_{31})\alpha^7
+(3a_{03}a_{12}-a_{03}a_{30}-5a_{03}b_{03}-
a_{12}a_{21}-3a_{12}b_{30}-3a_{21}b_{21}-3a_{30}b_{12}\\-3b_{03}b_{12}+b_{03}b_{30}+b_{12}b_{21}-5b_{21}b_{30}+2a_{13}+a_{31}-b_{04}+b_{22}+b_{40})\alpha^6+(2a_{03}a_{21}\\-4a_{03}b_{12}+a_{12}^2-2a_{12}a_{30}
-4a_{12}b_{03}-a_{21}^2-4a_{21}b_{30}-4a_{30}b_{21}+3b_{03}^2-2b_{03}b_{21}-b_{12}^2\\+
2b_{12}b_{30}+b_{21}^2-3b_{30}^2+a_{04}+2a_{22}+a_{40}-b_{31}+b_{31})\alpha^5+(a_{03}a_{30}-3a_{03}b_{21}+a_{12}a_{21}\\-3a_{12}b_{12}
-3a_{21}a_{30}-3a_{21}b_{03}-5a_{30}b_{30}+5b_{03}b_{12}-b_{03}b_{30}
-b_{12}b_{21}+3b_{21}b_{30}+a_{13}\\+2a_{31}-b_{04}-b_{22}+b_{40})\alpha^4+(-2a_{03}b_{30}-2a_{12}b_{21}-2a_{21}b_{12}-2a_{30}^2-2a_{30}b_{03}+4b_{03}b_{21}\\+2b_{12}^2+2b_{30}^2+a_{22}+2a_{40}-b_{13}-b_{31})\alpha^3+(-a_{12}b_{30}-a_{21}a_{30}-a_{21}b_{21}-a_{30}b_{12}+3b_{03}b_{30}\\+3b_{12}b_{21}+b_{21}b_{30}
+a_{31}
-b_{22}-b_{40})\alpha^2+(-a_{30}^2+2b_{12}b_{30}+b_{21}^2+b_{30}^2+a_{40}-b_{31})\alpha\\+(a_{30}b_{30}+b_{21}b_{30}-b_{40})
=0;
\end{array}
}
\end{equation}
\begin{equation}\label{ButterflyDirections_hyp}
{\tiny
\begin{array}{l}
a_{03}^2\alpha^{10}+(-2a_{03}a_{21}-6a_{03}b_{03}-a_{12}^2+4a_{04})\alpha^8+
(-4a_{03}a_{30}-4a_{03}b_{12}-4a_{12}a_{21}-4a_{12}b_{03}\\+4a_{13})\alpha^7
+(-2a_{03}b_{21}-6a_{12}a_{30}-2a_{12}b_{12}-3a_{21}^2-2a_{21}b_{03}+5b_{03}^2+4a_{22}-4b_{04})\alpha^6+\\(-8a_{21}a_{30}+8b_{03}b_{12}+4a_{31}-4b_{13})\alpha^5
+(2a_{12}b_{30}+2a_{21}b_{21}-5a_{30}^2+2a_{30}b_{12}+6b_{03}b_{21}\\+3b_{12}^2+4a_{40}-4b_{22})\alpha^4+(4a_{21}b_{30}+4a_{30}b_{21}+4b_{03}b_{30}+4b_{12}b_{21}-4b_{31})\alpha^3+(6a_{30}b_{30}+\\2b_{12}b_{30}+b_{21}^2-4b_{40})\alpha^2-b_{30}^2=0;
\end{array}
}
\end{equation}
\begin{equation}\label{ButterflyDirections_par}
{\tiny
\begin{array}{l}
(-a_{03}a_{12}-a_{03}b_{03}+a_{04})\alpha^7+(-2a_{03}a_{21}-a_{12}^2+b_{03}^2+a_{13}-b_{04})\alpha^6+(-3a_{03}a_{30}+a_{03}b_{21}\\-3a_{12}a_{21}
+a_{12}b_{12}+a_{21}b_{03}+b_{03}b_{12}+a_{22}-b_{13})\alpha^5
+(2a_{03}b_{30}-4a_{12}a_{30}+2a_{12}b_{21}-2a_{21}^2\\+2a_{21}b_{12}+2a_{30}b_{03}+a_{31}-b_{22})\alpha^4
+(3a_{12}b_{30}-5a_{21}a_{30}+3a_{21}b_{21}+3a_{30}b_{12}-b_{03}b_{30}-b_{12}b_{21}\\+a_{40}-b_{31})\alpha^3+(4a_{21}b_{30}
-3a_{30}^2+4a_{30}b_{21}-2b_{12}b_{30}-b_{21}^2-b_{40})\alpha^2+b_{30}(5a_{30}-3b_{21})\alpha\\-2b_{30}^2=0. 
\end{array}
}
\end{equation}
At elliptic (resp. hyperbolic) points, the singularity of $P_{\pi^e_i}$, $i=1,\ldots,10$ {\rm (}resp. $P_{\pi^h_j}$, $j=1,\ldots,10${\rm )} is of Type 6 {\rm (}butterfly{\rm)} or worse. 
There is a curve on the surface $M$ {\rm (}which may be empty{\rm )} where the singularity 
of $P_{\pi^e_i},i=1,\ldots,10,$ {\rm (}resp. $P_{\pi^h_j}$, $j=1,\ldots,10${\rm )} becomes of Type 7. There is also another curve {\rm (}possibly empty{\rm )}
on $M$ where the singularity becomes of Type 8  and possible isolated points on this curve where it degenerates further to Type 9 or Type 10.

At  parabolic points, the singularity of $P_{\pi^p_k}$, $k=1,\ldots,7$ is of Type 6 {\rm (}butterfly{\rm)} or worse and happen along the parabolic curve. 
There is an isolated point on the parabolic set {\rm (}which may be empty{\rm )} where the singularity 
of $P_{\pi^p_k}$, $k=1,\ldots,7$ becomes of Type 7. There is also another isolated point {\rm (}possibly empty{\rm )} on the parabolic set where the singularity becomes of Type 8. 

\item[(a5)] Generically, at elliptic {\rm (}resp. hyperbolic{\rm )} points there is a curve on the surface $M$ when the discriminant of Equation {\rm {\rm (\ref{ButterflyDirections})}} {\rm (}resp. {\rm {\rm (\ref{ButterflyDirections_hyp})}}{\rm )} vanishes which we call {\rm butterfly elliptic curve} {\rm (}resp. {\rm butterfly hyperbolic curve}{\rm )}. There are isolated points at parabolic curve when the discriminant of Equation {\rm {\rm (\ref{ButterflyDirections_par})}} vanishes which we call {\rm butterfly parabolic points}.
\end{itemize}

\end{theorem}

\begin{proof}
Consider $P_\pi$ as in \eqref{Proj1} with $\beta=1$. For elliptic, hyperbolic, and parabolic points, we may further simplify the quadratic part $j^2(f_1,f_2)=(Q_1,Q_2)$, as in Table \ref{Gibson}, by taking
\[
(Q_1, Q_2) = (x^2 - y^2, xy), \quad (x^2, y^2), \quad \text{and} \quad (xy, y^2),
\]
respectively. Here we write $g = g_{(\alpha,\lambda,1,\mu)}$.
The singularities we are seeking in Table \ref{tab:rieger} depend on certain jets of $g$, which we compute using Maple.
We prove the case when $p$ is an elliptic point; the other cases are proved analogously and are omitted.

\

\noindent{\it Elliptic case:} The task then becomes a problem of recognition of singularities
of map-germs from the plane to the plane. We use the criteria in Table 6.1 in \cite{Kabata,IFRT-book}.

(a1) We find that 
$$
j^2g=(1-\alpha^2-\alpha\mu)x^2-(2\alpha+\mu) xy-y^2.
$$

Therefore, the singularity of $P_{\pi}$ is a fold unless the coefficient of $x^2$ in  $j^2g$ is zero, 
that is, $1-\alpha^2-\alpha\mu=0$, i.e, $(\alpha+(1/2)\mu)^2-(1/4)\mu^2=1$. This is the equation of the surface $\mathcal S$ in $\Pi$, which is the product of a hyperbola with a line.

(a2) On $\mathcal S$, we can write $\mu=-\displaystyle\frac{\alpha^2-1}{\alpha}$. Then, the coefficient of $xy$ in 
$j^2g$ becomes $-\displaystyle\frac{\alpha^2+1}{\alpha}$ and is never zero, so we get singularities in Table \ref{tab:rieger} with a 2-jet 
$\mathcal A^{(2)}$-equivalent to $(y,xy)$, that is, $g$ has a regular singular set (see Proposition \ref{prop:class}(i)).

For $\pi\in \mathcal S$, the singularity is a cusp (Type 3) unless the coefficient of $x^3$ in the Taylor expansion of $g$ is zero, that is, \begin{equation}\label{cuspcondition}
\begin{array}{lcl}
\Lambda&=&-\alpha(\alpha^2+1)\lambda+\alpha^5b_{03}+(a_{03}+b_{12})\alpha^4+(a_{12}-b_{03}+b_{21})\alpha^3\\
&&+(a_{21}-b_{12}+b_{30})\alpha^2+(a_{30}-b_{21})\alpha-b_{30}=0.
\end{array}
\end{equation}
If this happens, 
we can solve for $\lambda$ because $\alpha\ne 0$ on $S$ and get a curve $C\subset S$, parametrised by $(\alpha, \lambda(\alpha),\mu(\alpha))$, where the
singularity of $P_{\pi}$ is more degenerate than cusp.

(a3) For $\pi\in C$, the singularity is a swallowtail (Type 5) unless the coefficient of $x^4$ in the Taylor expansion of 
$g$ is zero, that is,  Equation (\ref{ButterflyDirections}).
Generically 
 the zeros of Equation  (\ref{ButterflyDirections}) has up to ten real distinct solutions $\alpha_i,i=1,\ldots,10$ which determine planes $\pi'_i, i=1,\ldots,10$ on $C$. 

 (a4)  The singularity of $P_{\pi'_i}$, $ i=1,\ldots,10$, is of Type 6 or 7 
when the coefficients $D_1(\alpha_i)$ of $x^5$ in the Taylor expansion of $g$ is not zero, where $D_1(\alpha)$ is a polynomial in $\alpha$ of degree 18.
In this case, according to the criteria in Table 6.1 in \cite{IFRT-book}, and denoting by $g_{ki}$ the coefficient of $x^{k}y^i$ 
in the Taylor expansion of $g$, the singularity is of Type 7 when
\begin{equation*}
\label{eq:Type7}
(8g_{50}g_{70}-5g_{60}^2)g_{11}+2g_{50}(g_{21}g_{60}-20g_{31}g_{50})g_{11}+35g_{21}^2g_{50}^2=0,
\end{equation*}
otherwise, it is of Type 6. 
The Type 7 singularities occur along a regular curve on $M$, with equation $D_2(\alpha_i)x+D_3(\alpha_i)y+O(2)=0$, where $D_2(\alpha)$ and $D_3(\alpha)$ are polynomial in $\alpha$.  There are also some isolated points on the curve where $D_2(\alpha_i)=0$ or $D_3(\alpha_i)=0$. 

 Suppose now that $D_1(\alpha_i)=0$. 
For a generic surface $M$, the resultant, with respect to $\alpha$, 
of $D_1(\alpha)$  and of the left hand side of Equation (\ref{ButterflyDirections}) 
vanishes along a regular curve on $M$.
On this curve, the singularity of $P_{\pi'_i}$ is of Type 8 when 
\begin{equation*}
\label{eq:Type8}
g_{60}\neq0 \; \mbox{ and \;$\Lambda_1=(5g_{60}g_{80}-3g_{70}^2)g_{11}^2+g_{60}(g_{21}g_{70}-30g_{31}g_{60})g_{11}+27g_{21}^2g_{60}^2\neq0$.}
\end{equation*}
It has isolated points on this curve 
where it becomes of Type 9 (if $\Lambda_1=0$) or Type 10 (if $g_{60}=0$).

(a5) The butterfly directions are determined by the solution of Equation (\ref{ButterflyDirections}). Denoting by $\xi$
its discriminant, then $\xi=0$ is generically a regular curve on surface $M$.

\qed
\end{proof}

In Section \ref{Apparent contour}, we give a geometric interpretation of this theorem in terms of the geometry of the apparent contour.

\begin{figure}[htb]
\includegraphics[ width=8.5cm, height=5cm]{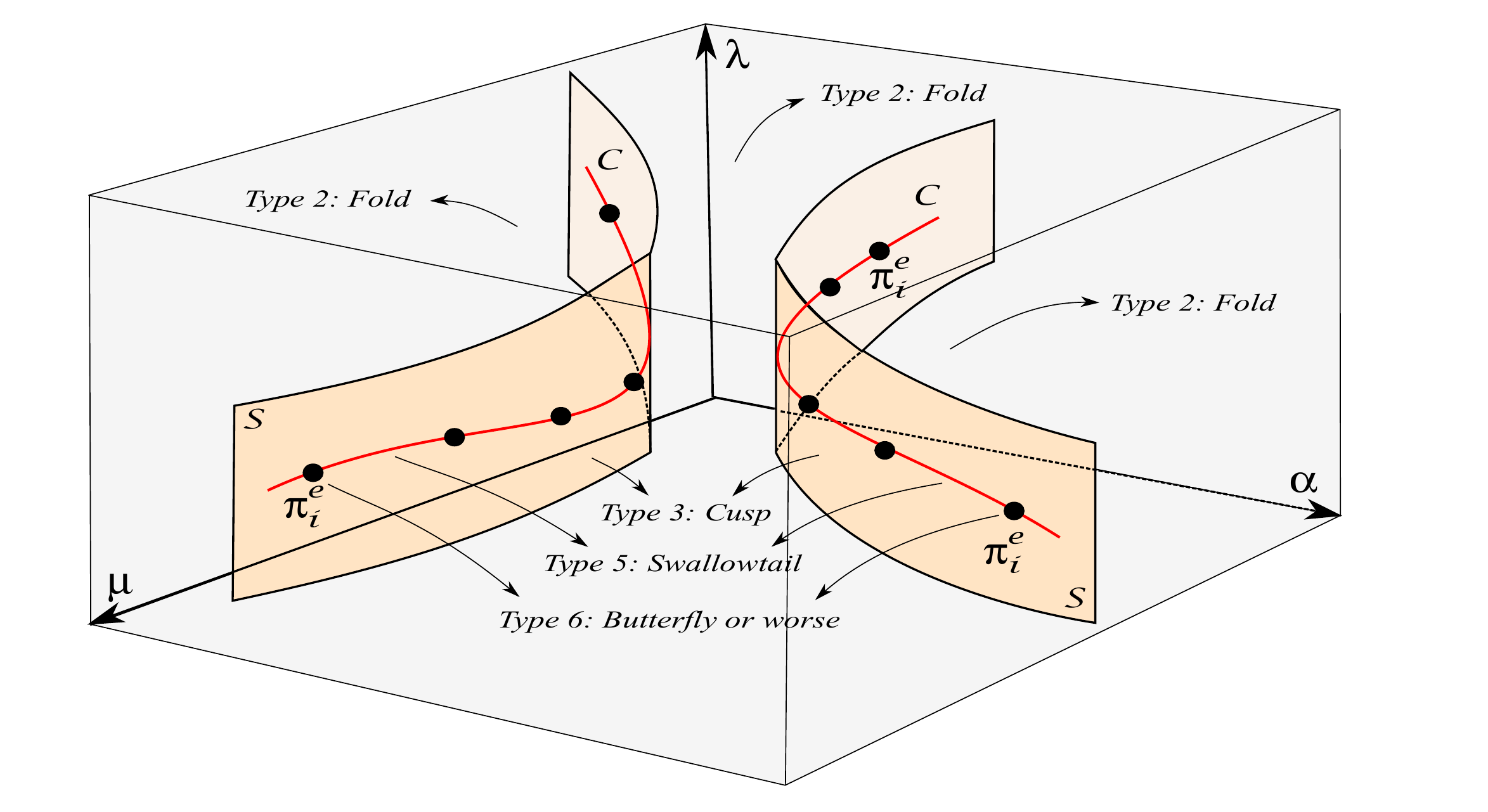}
\hspace{-1cm}
\caption{Stratification of the space of planes $\Pi$ according to the singularities of $P_\pi$ at elliptic points. In the gray region, $P_\pi$ is a fold. In the cream-colored region, $P_\pi$ is a cusp. The red curve corresponds to swallowtail singularities, while at the isolated black points on the red curve, $P_\pi$ has a butterfly singularity or worse.}
\label{fig:NewSpecialC}
\end{figure}

\subsubsection{${\bf u}$ is an asymptotic direction in $\pi$}\label{HyperbolicSec} In this case, the assumptions of Remark \ref{rem:class} (ii) are satisfied, where ${\bf u}=(1,\alpha,0,0)$ is an asymptotic direction. We analyze the situations in which $p$ is a hyperbolic, parabolic, or inflection point.


\subsubsection*{Hyperbolic case:}
We begin with the case where $p$ is a hyperbolic point. In this situation, we may assume that $M$ is given as in (\ref{Monge}) with $j^2(f_1,f_2)=(x^2,y^2)$. Since we are interested in the case where $j^2P_\pi\sim_{\mathcal A^{(2)}} (y,0)$, the conditions $\Omega({\bf u})=\mathcal B({\bf w})=0$ must hold. Hence, it follows that $\alpha=\beta=0$ and $\mu\neq0$ in expression \eqref{Proj1} of $P_{\pi}$. Without loss of generality, we can take $\mu =1$, therefore, the set $\Pi\subset Gr(2,4)$ 
contains a line $L$, parametrized by $\lambda$, along which $j^2P_\pi\sim_{\mathcal A^{(2)}} (y,0)$.
Moreover, the singular set $\Sigma_\pi$ has a Morse $A_1^\pm$-singularity if and only if 
\begin{equation}\label{MorseCondition}
4\lambda^2+4b_{21}\lambda+(b_{21}^2-3b_{12}b_{30})
\end{equation}
is non-zero.

\begin{theorem} \label{theo:hyp}

Let $M$ and the line $L$ be as above. Suppose that $p$ is a hyperbolic point.

\begin{itemize}
\item[(a1)]  For $\pi\in L\backslash\{\pi_1, \pi_2\}$ and away from a curve $\mathcal C_{4_2}$ on $M$ {\rm(}see {\rm(a4)}{\rm)}, the singularity of the map-germ $P_{\pi}$ is of Type $4_{2}^{\pm}$ {\rm(}beaks and lips{\rm)}. In particular, $\Sigma_\pi$ has a Morse $A_1^\pm$-singularity. 
 
\item[(a2)]  The planes $\pi_1$ and $\pi_2$ in {\rm(a1)} are determined by the roots of Equation~\eqref{MorseCondition}, provided that its discriminant is nonzero. The singularity of $P_{\pi_i}$, $i=1,2$, is of Type $4_{3}$ {\rm(}goose{\rm)}. In particular, $\Sigma_\pi$ has $A_{\geq2}$-singularity.  

\item[(a3)] There exists a curve on the surface $M$, called the {\normalfont $4_4$-curve}, along which $P_{\pi_i}$, for $i=1,2$, has a singularity of Type $4_{4}$ {\rm(}ugly goose{\rm)}. Moreover, there exists an isolated point on the $4_{4}$-curve where the singularity of $P_{\pi_i}$ becomes of Type $4_5$.

\item[(a4)] Generically, the discriminant of Equation~\eqref{MorseCondition} vanishes along 
a curve $\mathcal C_{4_2}$ on $M$. For $\pi\in L\backslash\{\pi'\}$ and along the curve $\mathcal C_{4_2}$, the singularity of $P_\pi$ is of Type $4_{2}^{\pm}$. Along this curve, the singularity of $P_{\pi'}$ is of Type $4_{3}$ and becomes of Type $4_4$ at isolated points on $\mathcal C_{4_2}$. (see Proposition \ref{condSingA3}  for a geometric interpretation)

 \item[ (b1)] 

For $\pi\in L\backslash\{\pi_i\}$, $i=3,\ldots,7$, and along a curve on the surface $M$, called the {\normalfont $11_5$-curve}, the singularity of $P_{\pi}$ is of Type $11_{5}$ {\rm(}gulls{\rm)}. The planes $\pi_6$ and $\pi_7$ are determined by the roots of
\begin{equation}\label{11_k}
\begin{array}{l}
(2a_{30}b_{21}^2+b_{12}b_{21}^2-8a_{30}b_{40}-2b_{21}b_{31}+4b_{50})\lambda^2+\\
(a_{30}b_{21}^3-4a_{30}b_{21}b_{40}+4b_{12}b_{21}b_{40}-b_{21}^2b_{31}+4b_{21}b_{50}-4b_{31}b_{40})\lambda+\\
4b_{12}b_{40}^2+b_{21}^2b_{50}-2b_{21}b_{31}b_{40},
\end{array}
\end{equation}
provided that its discriminant is nonzero. For $\pi=\pi_i$, $i=6,7$, the singularity of $P_{\pi_i}$ is of Type $11_{7}$ {\rm(}ugly gulls{\rm)}. Moreover, there exists an isolated point on the $11_5$-curve where the singularity of $P_{\pi_i}$ becomes of Type $11_{9}$. In particular, $\Sigma_\pi$ has a Morse $A_1^\pm$-singularity.

\item[ (b2)] The planes $\pi_4$ and $\pi_5$, are determined by zeros of 
$
\lambda^2+b_{21}\lambda+b_{40}
$
provided that its discriminant is nonzero.
The singularity of $P_{\pi_i}$, $i=4,5$ is of Type $12$, $13$, or $15$, where singularities of Type $13$ and $15$ occur at isolated points on the $11_5$-curve. In particular, $\Sigma_\pi$ has a Morse $A_1^\pm$-singularity. 

\item[ (c1)] For the plane $\pi_3$, the singularity of $P_{\pi_3}$ is of Type $16$, $17$, $18$, or $19$, where singularities of Type $16$ occur along the $11_5$-curve, while singularities of Type $17$, $18$, and $19$ occur at isolated points on this curve.  In particular, $\Sigma_\pi$ has $A_{\geq2}$-singularity.  
\end{itemize}
\end{theorem}

\begin{proof}
We can make changes of coordinates in the source so that $P_{\pi}(x,y)\sim_{\mathcal A} (y,g_{(0,\lambda,1)}(x,y)).$
We have
$$
j^3g=b_{30}x^3+(b_{21}+2\lambda)x^2y+b_{12}xy^2+b_{03}y^3
$$

\noindent {\rm(a1)} If we denote by $g_{ij}$ the coefficient of $x^{i}y^i$ in $g$, then the condition for $g$ to have a Type $4^\pm_2$ (lips/beaks) singularity is 
$$
\left\{
\begin{array}{l}
g_{30}=b_{30} \neq 0\\ 
g_{21}^2-3g_{12}g_{30}=4\lambda^2+4b_{21}\lambda+(b_{21}^2-3b_{12}b_{30})\neq0.
\end{array}
\right.
 $$
Note that Equation ~\eqref{MorseCondition} and $g_{21}^2-3g_{12}g_{30}$ coincides, therefore, when the discriminant of this equation is non-zero (i.e., $b_{12}\neq0$), then away from two distinct planes and a curve $\mathcal C_{4_2}$ on $M$
, the singularity is of Type $4^\pm_2$. In particular, under this condition, $\Sigma_\pi$ has a Morse singularity of type $A_1^\pm$.

\noindent {\rm(a2)} For each one these planes, i.e. when $4\lambda^2+4b_{21}\lambda+(b_{21}^2-3b_{12}b_{30})=0$, the singularity of $P_{\pi}$ is Type $4_3$ (goose) unless 
\begin{equation}\label{4_4-curve}
27g_{13}g_{30}^3+9g_{21}^2g_{30}g_{31}-18g_{21}g_{22}g_{30}^2-4g_{21}^3g_{40}=0.
\end{equation}
This expression is a polynomial of degree 4 in the variable $\lambda$. Moreover, since the discriminant of the quadratic terms of $g_x$ is given by zeros of Equation ~\eqref{MorseCondition}, 
it follows that $\Sigma_\pi$ has an $A_2$-singularity or worse.

\noindent {\rm(a3)} When the resultant of Equation~\eqref{4_4-curve} and $g_{21}^2-3g_{12}g_{30}$ is non-zero, there is a regular curve on $M$  where the singularity of $P_{\pi}$ is $4_4$ if 
$$
\begin{array}{l}
81g_{14}g_{30}^5+5g_{21}^4g_{30}g_{50}-12g_{21}^4g_{40}^2-12g_{21}^3g_{30}^2g_{41}+36g_{21}^3g_{30}g_{31}g_{40}-36g_{21}^2g_{22}g_{30}^2g_{40}\\
+27g_{21}^2g_{30}^3g_{32}-27g_{21}^2g_{30}^2g_{31}^2+54g_{21}g_{22}g_{30}^3g_{31}-54g_{21}g_{23}g_{30}^4-27g_{22}^2g_{30}^4\neq0.
\end{array}
$$
In the opposite case, the singularity of $P_{\pi}$ is of Type $4_5$, and it happens at an isolated point on $M$.

\noindent {\rm(a4)} 
The discriminant of Equation~\eqref{MorseCondition} vanishes when $b_{12}=0$. On a generic surface $M$, this defines the regular curve $\mathcal C_{4_2}$. In this case,
$
g_{21}^2-3g_{12}g_{30}=(2\lambda+b_{21})^2.
$
Hence, when $2\lambda+b_{21}\neq0$, the singularity is of Type $4^\pm_2$ away from a distinguished plane $\pi'$. For this plane, given by $\lambda=-\frac{b_{21}}{2}$, Equation~\eqref{4_4-curve} implies that the singularity is of Type $4_3$ whenever $b_{13}-b_{21}a_{12}\neq0$. Otherwise, the singularity is of Type $4_4$, occurring at isolated points of $\mathcal C_{4_2}$.

 The singularity of $P_{\pi}$ is of Type $11_{2k+1}$ when 
$$
\left\{
\begin{array}{l}
g_{30} =b_{30}=0;\\ 
g_{21}=2\lambda+b_{21}\neq0;\\ 
g_{40}=\lambda^2+b_{21}\lambda+b_{40}\neq0.
\end{array}
\right.
 $$
\noindent {\rm(b1)} If $g_{21}^2g_{50}+4g_{12}g_{40}^2-2g_{21}g_{31}g_{40}\neq0$, i.e., Equation \eqref{11_k} is non-zero, 
then away from five  planes (determined by zeros of $g_{21}$, $g_{40}$ and Equation (\ref{11_k})) and along to the regular $11_5$-curve on $M$, 
the singularity is of Type $11_{5}$ (gulls).  When Equation (\ref{11_k}) is zero and its discriminant is non-zero,  there exist two distinct planes such that the singularity  is of Type $11_{7}$ (ugly gulls) if the condition involving the coefficients of order $7$ of $g$, which we call $D$, is not zero, where
\begin{equation}\label{11_7}
\begin{array}{l}
D=128g_{12}^3g_{40}^4-112g_{12}^2g_{21}g_{31}g_{40}^3-12g_{12}g_{21}^4g_{40}g_{60}+16g_{12}g_{21}^3g_{40}^2g_{41}\\-16g_{12}g_{21}^2g_{22}g_{40}^3
+32g_{12}g_{21}^2g_{31}^2g_{40}^2+8g_{13}g_{21}^3g_{40}^3-g_{21}^6g_{70}+3g_{21}^5g_{31}g_{60}\\+2g_{21}^5g_{40}g_{51}-4g_{21}^4g_{31}g_{40}g_{41}
-4g_{21}^4g_{32}g_{40}^2+4g_{21}^3g_{22}g_{31}g_{40}^2-3g_{21}^3g_{31}^3g_{40}.
\end{array}
\end{equation}
 The singularity becomes of Type $11_9$ if $D=0$, and this occurs at an isolated point on the $11_5$-curve. In particular, as $g_{21}\ne0$, Equation \eqref{MorseCondition} is nonzero, then $\Sigma_\pi$ is a Morse singularity.


\noindent {\rm(b2)} When $g_{21}\neq0$ and $g_{30}=g_{40}=0$, then the discriminant of $g_{40}=0$ never vanish, i.e., $b_{21}^2-4b_{40}>0$, thus there exist exactly two planes where the singularity along the $11_5$-curve is of Type $12$ if  $2g_{21}g_{60}-5g_{31}g_{50}\neq0$. In the opposite case, the singularity is of Type $13$, and it occurs at an isolated point on the $11_5$-curve. In the same way, if a condition involving the coefficients of order $5$ of $g$ is zero, then at an isolated point on the  $11_5$-curve, the singularity is of Type $15$. 

\noindent {\rm(c1)} Finally, when $g_{30}=g_{21}=0$ and $g_{12}\neq0$, i.e., $\lambda=-\frac{b_{21}}{2}$ and $b_{12}\neq0$, there exists a plane $\pi_1$ such that the singularity of $P_{\pi_1}$ is Type $16,17,18$, and $19$ where Type $16$ occurs along the $11_5$-curve if $\xi=2g_{12}g_{50}-8g_{22}g_{40}+3g_{31}^2\neq 0$ and $4b_{40}- b_{21}^2\neq0$. We obtain a singularity of Type $17$ (resp. $18$) at isolated points on this curve when $\xi=0$ (resp. $4b_{40}-b_{21}^2=0$). The Type $19$ happens at isolated points, if $b_{12}=0$. Finally, since $b_{30}=0$ and $\lambda=-\frac{b_{21}}{2}$, Equation~\eqref{MorseCondition} vanishes. Hence, $\Sigma_\pi$ has an $A_{\geq2}$-singularity.\qed
\end{proof}

\

Now we present some geometric interpretation of Theorem \ref{theo:hyp}. Following from the duality in \cite{bruce-nogueira} between the
bifurcation set of the family of height functions with that of the family of orthogonal projections, that when we consider the orthogonal projection of $M$ in Monge form as (\ref{Monge}) at a hyperbolic point $p$ along the asymptotic direction ${\bf u}  = (1,0,0,0)$, so that this projection is given by
\[
P_{\bf u}(x,y) = \bigl(y, f_1(x,y), f_2(x,y)\bigr).
\]
When the $j^3P_{\bf u}(x,y)$ is $\mathcal{A}^{(3)}$-equivalent to
$\bigl(y, x^2, b_{30}x^3 + b_{12}xy^2 \bigr),$
the singularity is of Type $B_k$, $k \geq 1$, if and only if $b_{12} \neq 0$, and of Type $B_{\geq 2}$ if, in addition, $b_{30} = 0$. It is of Type $S_k$, $k \geq 2$, if and only if $b_{30} \neq 0$ and $b_{12} = 0$. Finally, it is of Type $C_3$ when $b_{30} = b_{12} = 0$. 

The binormal direction associated to ${\bf u}$ is ${\bf w} = (0,0,0,1)$, and the height function along ${\bf w}$ is given by
\[
h_{\bf w} = y^2 + b_{30}x^3 + b_{21}x^2y + b_{12}xy^2 + b_{03}y^3+O(4).
\]
It has an $A_{\geq 3}$-singularity (resp. $A_{\geq 4}$-) if and only if $b_{30} = 0$ (resp. $b_{30} =b_{21}^2-4b_{40}= 0$). 



Recall that $S_2$-curve and $B_2$-curve are the closure of the set of points $p$ on $M$ for which there exists a projection $P_{\bf u}$ having an $S_2$ (resp. $B_2$)-singularity at $p$. Since the $11_5$-curve (resp. $4_4$-curve and curve $\mathcal C_{4_2}$) from Theorem \ref{theo:hyp} can be seen as the closure of the set of points $p \in M$ for which there exists a projection $P_\pi$ having an $11_5$ (resp. $4_4$ and $4_2$)-singularity or worse at $p$, Thus, the following result is a direct consequence of Theorem~\ref{theo:hyp} (b1).

\begin{proposition}\label{condSingA3} Let $M$ be a generic surface in $\mathbb R^4$. With the conditions of Theorem \ref{theo:hyp} and the above comments, the following holds: 
\begin{itemize}
\item[(i)] The $11_5$-curve coincides with the $B_2$-curve. 
\item[(ii)] The $A_4$-singularity of $h_{\bf w}$ is related to the singularity of Type 18 of $P_\pi$ and the $C_3$-singularity of $P_{\bf u}$ with Type 19 of $P_\pi$.
\item[(iii)] The curve $\mathcal C_{4_2}$  coincides with the $S_2$-curve. The $S_3$-singularity of $P_{\bf u}$ is related to singularity of Type $4_4$ of $P_\pi$. In particular, the $4_4$-curve intersects the $S_2$-curve transversally at an $S_3$-singularity of $P_{\bf u}$.
\end{itemize}
\end{proposition}
\begin{proof}
The proofs of (i) and (ii) follow from the fact that the singularity of $P_{\pi}$ is of Type $11_{5}$ or worse when $b_{30} = 0$, and it is Type 18 and 19 if $b_{30} =b_{21}^2-4b_{40}= 0$, and $b_{30}=b_{12}=0$, respectively. For (iii), the curve $\mathcal C_{4_2}$ coincides with the $S_2$-curve, since the condition $b_{21}=0$ defines $S_{\geq 2}$-singularities of $P_{\bf u}$. In particular, the $S_3$-singularity of $P_{\bf u}$ is an isolated point on this curve, if $b_{30} \neq 0$ and 
\[
b_{12} = b_{13}-b_{21}a_{12}=0.
\]
Substituting these conditions into the criteria for a singularity of Type $4_4$ of $P_{\pi_i}$, considering the plane $\pi_i$ determined by $\lambda=-\frac{b_{21}}{2}$, we obtain that
\[
b_{30} \neq 0, 
\qquad
g_{21}^2-3g_{12}g_{30}=(2\lambda+b_{21})^2=0,
\]
and Equation~\eqref{4_4-curve} is satisfied. 
It follows that the $4_4$-curve (determined by Equation~\eqref{4_4-curve}) intersects the $S_2$-curve transversally at an $S_3$-singularity of $P_{\bf u}$.
\qed
\end{proof}

\begin{remark}
At hyperbolic points, we define the $11_{5}$-curve as the closure of the set of points $p \in M$ for which there exists a projection $P_\pi$ having an $11_5$-singularity or worse at $p$. By Theorem~\ref{theo:hyp}, this curve can equivalently be described as the locus of points where $P_{\pi}$ has a singularity of Type $11_7$, $12$, or $16$.

\end{remark}

\

\subsubsection*{Parabolic case:} Now we consider the case when $p$ is a parabolic point. Similar to the hyperbolic case, 
 we may assume that $M$ is given as in (\ref{Monge}) with $j^2(f_1,f_2)=(xy,y^2)$ with $\alpha=\beta=0$ and $\mu\neq0$ in expression \eqref{Proj1} of $P_{\pi}$. Again, we can take $\mu =1$ and the set $\Pi\subset Gr(2,4)$ 
 contains a line $L$, parametrized by $\lambda$, along which $j^2P_\pi\sim_{\mathcal A^{(2)}} (y,0)$.
In particular, the singular set $\Sigma_\pi$ has a Morse $A_1^\pm$-singularity if and only if 
\begin{equation}\label{MorseConditionPar}
6b_{30}\lambda+(3b_{12}b_{30}-b_{21}^2)
\end{equation}
is non-zero.


\begin{theorem} \label{theo:inf}

Let $M$ and $\Pi$ be as above, and suppose that $p$ is a parabolic point.

\noindent 
\begin{itemize}

\item[(a1)]  For $\pi\in L\backslash\{\pi_1\}$ and away from an isolated point  at $M$ {\rm(}see {\rm(a3)}{\rm)}, the singularity of the map-germ $P_{\pi}$ is of Type $4_{2}^{\pm}$ {\rm(}beaks and lips{\rm)}. In particular, $\Sigma_\pi$ has a Morse $A_1^\pm$-singularity. 
 
\item[(a2)]  The plane $\pi_1$ in {\rm(a1)} is determined by the root of Equation~\eqref{MorseConditionPar}. The singularity of $P_{\pi_1}$, $i=1$, is of Type $4_{3}$ {\rm(}goose{\rm)}. In particular, $\Sigma_\pi$ has $A_{\geq2}$-singularity.  

\item[(a3)] There exists an isolated point at the surface $M$, along which $P_{\pi_1}$ has a singularity of Type $4_{4}$ {\rm(}ugly goose{\rm)}. 

 

 
\item[ (b1)] Away from the plane $\pi_2$, and along an isolated point of the parabolic curve on the surface $M$,
the singularity of $P_{\pi}$ is of Type $11_{5}$ {\rm(}gulls{\rm)}. For $\pi=\pi_2$ the singularity of $P_{\pi}$ is of Type $11_{7}$ (ugly gulls).  
\end{itemize}
\end{theorem}

\begin{proof}
(a1) 
We consider $P_{\pi}(x,y)\sim_{\mathcal A} (y,g_{(0,\lambda,1)}(x,y)),$ then
$$
j^3g=b_{30}x^3+b_{21}x^2y+(b_{12}+2\lambda)xy^2+b_{03}y^3.
$$
If we denote by $g_{ij}$ the coefficient of $x^{i}y^i$ in $g$, then the condition for $g$ to have a lips/beaks singularity is 
$$
\left\{
\begin{array}{l}
g_{30}=b_{30} \neq 0\\ 
g_{21}^2-3g_{12}g_{30}=-6b_{30}\lambda+(b_{21}^2-3b_{12}b_{30})\neq0.
\end{array}
\right.
 $$
Thus, as $b_{30}\neq0$, then away from a plane $\pi_1$ determined by $\lambda=\frac{(b_{21}^2-3b_{12}b_{30})}{6b_{30}}$, the singularity is of Type $4_2^\pm$.  In particular, under this condition, $\Sigma_\pi$ has a Morse singularity of type $A_1^\pm$.

\noindent (a2) When $\lambda=\frac{(b_{21}^2-3b_{12}b_{30})}{6b_{30}}$, the singularity of $P_{\pi_1}$ is of Type $4_3$ (goose) unless 
$$
\begin{array}{l}
54a_{12}b_{12}b_{30}^3-18a_{12}b_{21}^2b_{30}^2-36a_{21}b_{12}b_{21}b_{30}^2+12a_{21}b_{21}^3b_{30}+18a_{30}b_{12}b_{21}^2b_{30}\\-6a_{30}b_{21}^4
+81b_{03}b_{12}b_{30}^3-27b_{03}b_{21}^2b_{30}^2-9b_{12}^2b_{21}b_{30}^2+3b_{12}b_{21}^3b_{30}-54b_{13}b_{30}^3\\-18b_{21}^2b_{30}b_{31}
+36b_{21}b_{22}b_{30}^2-4b_{21}^2b_{40}=0.
\end{array}
$$ 
\noindent (a3) In this case, there is an isolated point on $M$  where the singularity of $P_{\pi_1}$ is $4_4$ if 
$$
\begin{array}{l}
81g_{14}g_{30}^5+5g_{21}^4g_{30}g_{50}-12g_{21}^4g_{40}^2-12g_{21}^3g_{30}^2g_{41}+36g_{21}^3g_{30}g_{31}g_{40}-36g_{21}^2g_{22}g_{30}^2g_{40}\\
+27g_{21}^2g_{30}^3g_{32}-27g_{21}^2g_{30}^2g_{31}^2+54g_{21}g_{22}g_{30}^3g_{31}-54g_{21}g_{23}g_{30}^4-27g_{22}^2g_{30}^4\neq0.
\end{array}
$$

\noindent (b1) Now the singularity of $P_{\pi}$ is of Type $11_{2k+1}$ when $g_{30} =b_{30}=0$, 
$g_{21}=b_{21}\neq0$, $g_{40}=b_{40}\neq0.$
%
If $g_{21}^2g_{50}+4g_{12}g_{40}^2-2g_{21}g_{31}g_{40}\neq0$, i.e., 
\begin{equation}\label{11_7_par}
\begin{array}{l}
(b_{21}^2-4b_{40})(a_{30}b_{21}-2b_{40})\lambda+
4b_{12}b_{40}^2+b_{21}^2b_{50}-2b_{21}b_{31}b_{40}\neq0,
\end{array}
\end{equation}
then away from a plane $\pi_2$ (determined by $\lambda$), and along the an isolated point on $M$ ($b_{30}=0$), the singularity of $P_{\pi}$ is of Type $11_{5}$ (gulls).  When Equation (\ref{11_7_par}) is zero, the singularity of $P_{\pi_2}$ is of Type $11_{7}$ (ugly gulls) if Equation (\ref{11_7}) is non-zero. 





\qed
\end{proof}

\

Again following from the duality in \cite{bruce-nogueira} the orthogonal projection of $M$ in Monge form as (\ref{Monge}) at a parabolic point $p$ along the asymptotic direction ${\bf u}  = (1,0,0,0)$, so  
the $j^2 P_{\bf u}(x,y)$ is $\mathcal{A}^{(2)}$-equivalent to
$\bigl(y, xy, 0 \bigr).$ The singularity of $P_{\bf u}$ is of type $H_k$, $k=2,3$, if and only if $b_{30} \neq 0$, and of type $P_3(c)$ if $b_{30} = 0$. Furthermore, 
as the $H_2$-curve (defined in Section \ref{section2}) coincides with the parabolic set $\Delta$, then we have the following.

\begin{proposition}\label{CondP3}
Let $M$ be a generic surface in $\mathbb{R}^4$. The parabolic set (or $H_2$-curve) is the set of points where the projection $P_{\pi}$ has singularities of Type $4_2^\pm$ and $4_3$. The $P_3(c)$-singularity of $P_{\bf u}$ is related to the occurrence of singularities of Type $11_5$ and $11_7$ of $P_{\pi}$.
\end{proposition}
\begin{proof}
The proof follows from the fact that at a parabolic point $p$, the singularity of $P_{\pi}$ is of Type $4^\pm_4$ or worse when $b_{30} \neq 0$, and it is of Type $11_5$ or worse if $b_{30} = 0$.\qed
\end{proof}

\begin{remark}
At parabolic points, the $H_3$-singularity of $P_{\bf u}$ is in general not related to the singularities of $P_{\pi}$.
\end{remark}

We summarize the Proposition \ref{condSingA3} and \ref{CondP3} in Figure \ref{fig:NewSpecialC}.

\begin{figure}[htb]
\includegraphics[ width=8.5cm, height=5.5cm]{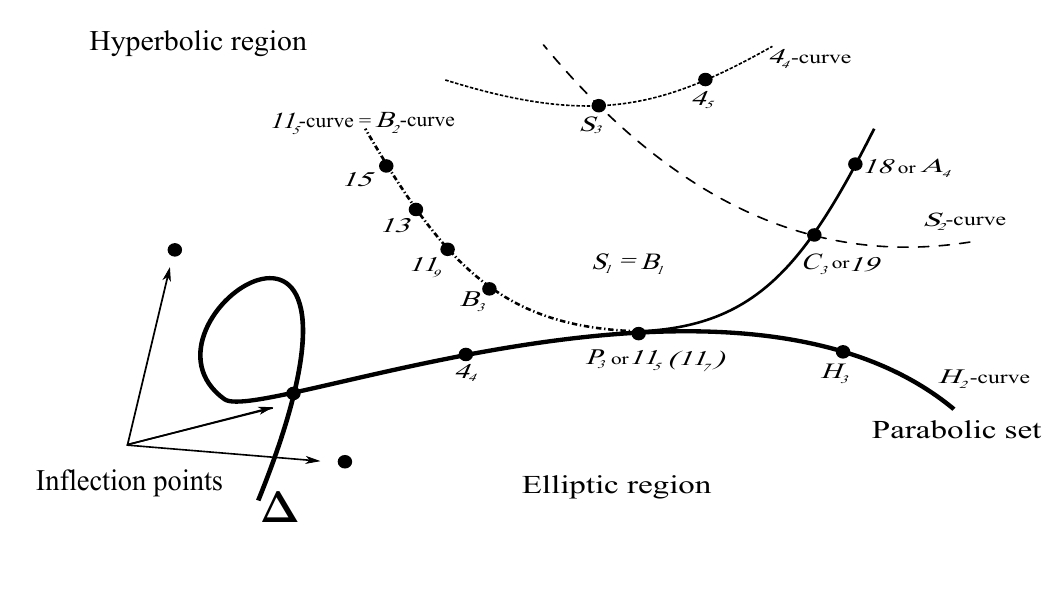}
\hspace{-1cm}
\caption{Curves and special points on $M$ via contact with 2-planes. In Figure \ref{fig:SpecialC}, we have added the $4_4$-curve to the hyperbolic region, together with the isolated points corresponding to singularities of Type $11_9$, $13$, $15$, $18$ and $4_5$. In the parabolic set $\Delta$, we have added the isolated point corresponding to a singularity of Type $4_4$.}
\label{fig:NewSpecialC}
\end{figure}

\subsubsection*{Inflection case:} Now we consider the case when $p$ is an inflection point. There are two types of inflection: imaginary type, where the pair of quadratics $(Q_1, Q_2)$
is equivalent to $(x^2 + y^2,0)$ (parabolic set is an isolated point); and the real type, when $(Q_1, Q_2)$ is equivalent to $(xy,0)$ (parabolic set is a crossing). We shall treat the cases separately. 
Since all directions in the tangent plane to $M$ at such a point are asymptotic, we may take ${\bf u}=(1,\alpha,0,0)$ as an asymptotic direction.

Here we can consider $\mu=1$ in Equation \eqref{Proj1} of $P_{\pi}$. 
In this cases, the set $\Pi\subset Gr(2,4)$ 
can be parametrised by $(\alpha,\beta,\lambda)\in \mathbb R^3$.

\begin{theorem}\label{theo:Inf}
Let $M$ and $\Pi$ be as above, and $p$ an inflection point. 
%
\begin{itemize}
\item[(A)] Suppose $j^2(f_1,f_2)=(x^2+ y^2,0)$.
\begin{itemize}
\item[(a1)] For $\pi \in \Pi$ and away from a surface $\mathcal S_1 \subset \Pi$, the singularity is of Type 2 {\rm (}fold{\rm)}. In particular, $\Sigma_\pi$ is a regular curve. 

\item[(b1)]  For $\pi \in \mathcal S_1$, except along the two curves $\mathcal C_1$ and $\mathcal C_2$, determined by two special directions in $T_pM$, the singularity of $P_{\pi}$ is of Type $4_2^\pm$ (lips/beaks).

\item[(b2)]  For $\pi\in\mathcal C_k$, $k=1,2$ and away from a plane $\pi_1$,  the singularity of $P_{\pi}$ is of Type $4_{3}$ {\rm(}goose{\rm)}. For $\pi_1\in\mathcal C_k$, the singularity of $P_{\pi_1}$ is of Type $4_4$.

 \item[ (c1)] For $\pi$ along the three curves $\mathcal C_k$, $k=3,4,5$ on $\Pi$, determined by three directions in $T_pM$, and away from three planes $\pi_i$, $i=2,3,4$, the singularity of $P_{\pi}$ is of Type $11_{5}$.

\item[(c2)] For $\pi=\pi_2$ or $\pi=\pi_3$, the singularity of $P_{\pi}$ is of Type $11_{7}$ (ugly gulls).  
For $\pi=\pi_4$ the singularity of $P_{\pi}$ is of \mbox{Type $12$}.


\end{itemize}
\item[(B)] Suppose $j^2(f_1,f_2)=(xy,0)$.
\begin{itemize}
\item[(a1)] For $\pi \in \Pi$ and away from union of two surfaces $\mathcal S_2, \mathcal S_3 \subset \Pi$, the singularity of the map-germ $P_{\pi}$ is of Type 2 {\rm (}fold{\rm)}.

\item[(a2)]  For $\pi\in \mathcal S_2\backslash \mathcal S_3$ and away from a curve $\mathcal C$ on $\mathcal S_2$, the singularity of $P_{\pi}$ is of Type 3 {\rm (}cusp{\rm)}. 

\item[(a3)] For $\pi \in \mathcal C$ except an exceptional plane $\pi_1$, the singularity of $P_{\pi}$ is of Type 5 (swallowtail).

\item[(a4)] For $\pi_1 \in \mathcal C$, the singularity of $P_{\pi}$ is of Type 6 (butterfly).

\item[(b1)]  For $\pi \in \mathcal S_3$, except along the two curves $\mathcal C_1$ and $\mathcal C_2$, determined by two directions in $T_pM$, the singularity of $P_{\pi}$ is of Type $4_2^\pm$ (lips/beaks).

\item[(b2)]  For $\pi\in\mathcal C_k$, $k=1,2$ and away from a plane $\pi_1$,  the singularity of $P_{\pi}$ is of Type $4_{3}$ {\rm(}goose{\rm)}. For $\pi_1\in\mathcal C_k$, the singularity of $P_{\pi_1}$ is of Type $4_4$.

 \item[ (c1)] For $\pi$ along the three curves $\mathcal C_k$, $k=3,4,5$ on $\Pi$, determined by three directions in $T_pM$, and away from a plane $\pi_i$, $i=2,3,4$, the singularity of $P_{\pi}$ is of Type $11_{5}$.

\item[(c2)] For $\pi=\pi_2$ or $\pi=\pi_3$, the singularity of $P_{\pi}$ is of Type $11_{7}$ (ugly gulls).  
For $\pi=\pi_4$ the singularity of $P_{\pi}$ is of \mbox{Type $12$}.
\end{itemize}
\end{itemize}
\end{theorem}

\begin{proof}
$(A)$ We may assume that
$P_{\pi}(x,y)\sim_{\mathcal A} (y, g_{(\alpha,\beta,\lambda)}(x,y)).$
The $2$-jet of $g$ is given by
\[
j^2g=\beta(\alpha^2+1)\,x^2 + 2\beta\alpha\,xy + \beta\,y^2.
\]
Since $\alpha^2+1\neq 0$, the singularity of $P_{\pi}$ is a fold unless $\beta=0$. This condition defines the surface $\mathcal S_1$ in $\Pi$. In this case, $j^2g\equiv 0$, and hence
we have
\begin{eqnarray*}
j^3g&=&(-\alpha^3 b_{03} - \alpha^2 b_{12} - \alpha b_{21} - b_{30})\,x^3
+ (-3\alpha^2 b_{03} - 2\alpha b_{12} - b_{21})\,x^2 y\\
&&+ (-3\alpha b_{03} - b_{12})\,x y^2
- b_{03}\,y^3.
\end{eqnarray*}
If we denote by $g_{ij}$ the coefficient of $x^{i}y^i$ in $g$, then the condition for $P_{\pi}$ to have a lips/beaks singularity is 
$$
\left\{
\begin{array}{l}
g_{30}=(b_{03}\alpha^3  + b_{12}\alpha^2  + b_{21}\alpha  + b_{30}) \neq 0\\ 
g_{21}^2-3g_{12}g_{03}=
(b_{12}^2-3 b_{03} b_{21})\,\alpha^2
+ (-9 b_{03} b_{30} + b_{12} b_{21})\,\alpha+b_{21}^2
- 3 b_{12} b_{30}  
\neq0.
\end{array}
\right.
$$
Thus, the condition for $P_\pi$ to have a $4_2^-$ (resp. $4_2^+$) singularity is $g_{30}\neq 0$ and $g^2_{21}-3g_{12}g_{30}>0$ (resp. $<0$).
Therefore, away from two curves $\mathcal C_1$ and $\mathcal C_2$ in $\Pi$ determined by zeros of 
$$
\Lambda_1(\alpha)=(-3 b_{03} b_{21} + b_{12}^2)\,\alpha^2
+ (-9 b_{03} b_{30} + b_{12} b_{21})\,\alpha
- 3 b_{12} b_{30} + b_{21}^2,$$ the singularity is of Type $4_2^\pm$. 

Generally, the discriminant of $\Lambda_1(\alpha)$ is non-zero, then if  the discriminant of $\Lambda_1(\alpha)>0$, there exist two curves $\mathcal C_k$, $k=1,2$, on $\mathcal S_1$ determined by two directions in $T_pM$ such that the singularity of $P_{\pi}$ for $\pi\in \mathcal C_i$ is Type $4_3$ (goose) unless 
$$
27g_{13}g_{30}^3+9g_{21}^2g_{30}g_{31}-18g_{21}g_{22}g_{30}^2-4g_{21}^2g_{40}=0.
$$ 
In this case, the above expression is an equation in variable $\lambda$ of order $1$, then there is an isolated point $\pi_1\in\mathcal C_i$ where the singularity of $P_{\pi_1}$ is $4_4$ if 
$$
\begin{array}{l}
81g_{14}g_{30}^5+5g_{21}^4g_{30}g_{50}-12g_{21}^4g_{40}^2-12g_{21}^3g_{30}^2g_{41}+36g_{21}^3g_{30}g_{31}g_{40}-36g_{21}^2g_{22}g_{30}^2g_{40}\\
+27g_{21}^2g_{30}^3g_{32}-27g_{21}^2g_{30}^2g_{31}^2+54g_{21}g_{22}g_{30}^3g_{31}-54g_{21}g_{23}g_{30}^4-27g_{22}^2g_{30}^4\neq0.
\end{array}
$$

Now the singularity of $P_{\pi}$ is of Type $11_{2k+1}$ when 
$$
\left\{
\begin{array}{l}
g_{30} =(\alpha^3 b_{03} + \alpha^2 b_{12} + \alpha b_{21} + b_{30})=0;\\ 
g_{21}=(-3\alpha^2 b_{03} - 2\alpha b_{12} - b_{21})\neq0;\\ 
g_{40}=-(\alpha^2 + 1)\,(3\alpha^2 b_{03} + 2\alpha b_{12} + b_{21})\,\lambda
- b_{04}\alpha^4
- b_{13}\alpha^3
- b_{22}\alpha^2
- b_{31}\alpha
- b_{40}\neq0.
\end{array}
\right.
 $$
In this case, according to the value of the discriminant of the cubic equation $g_{30}=0$, there may exist up to three special curves $\mathcal C_k$, $k=3,4,5$, on $\Pi$, determined by three directions in $T_pM$.


If $\Lambda_2(\lambda)=g_{21}^2g_{50}+4g_{12}g_{40}^2-2g_{21}g_{31}g_{40}=D_1\lambda^2+D_2\lambda+D_3\neq0$, then away from three planes $\pi_i\in\mathcal C_k$, $i=2,3,4$ (determined by   $\Lambda_2(\lambda)=0$ and $g_{40}=0$), the singularity of $P_{\pi}$ is of Type $11_{5}$ (gulls).  When $\Lambda_2(\lambda)=0$, the singularity of $P_{\pi_i}$, $i=2,3$ is of Type $11_{7}$ (ugly gulls) if Equation (\ref{11_7}) is not zero.
 
When $g_{40}=0$, there exists a plane $\pi_4$  where the singularity of $P_{\pi_4}$, is of Type $12$ if $2g_{21}g_{60}-5g_{31}g_{50}\neq0$. 
The case when $g_{30}=g_{21}=0$ and $g_{12}\neq0$ is not generic.  

$(B)$ If $j^2(f_1,f_2)=(xy,0)$ then the $2$-jet of $g$ is given by
\[
j^2g=\alpha\beta\,x^2 + \beta\,xy.
\]
 The singularity of $P_{\pi}$ is a fold unless $\alpha\beta=0$. This condition defines two surfaces, $\mathcal S_2$ and $\mathcal S_3$, in $\Pi$. We analyze separately the cases $\alpha=0$ (surface $\mathcal S_2$) and $\beta=0$ (surface $\mathcal S_3$). We focus on the first case, as the second is analogous and the statements are equivalent to those in case $(A)$.

(a2) On $\mathcal S_2$ with $\beta\neq0$, the $j^2g$ is $xy$, so we get singularities in Table \ref{tab:rieger} with a 2-jet 
$\mathcal A^{(2)}$-equivalent to $(y,xy)$. 
For $\pi\in \mathcal S_2$, the singularity is a cusp (Type 3) unless the coefficient of $x^3$ in the Taylor expansion of $g$ is zero, that is, 
$a_{30}\beta-b_{30}=0$.
If this happens, generically
we can solve for $\beta$ on $\mathcal S_2$ and get a curve $\mathcal C\subset \mathcal S_2$, parametrised by $\displaystyle(0, b_{30}/a_{30},\lambda)$, where the
singularity of $P_{\pi}$ is more degenerate than cusp.

(a3) For $\pi\in \mathcal C$, the singularity is a swallowtail (Type 5) unless the coefficient of $x^4$ in the Taylor expansion of 
$g$ is zero, that is, $a_{30}b_{30}\lambda-(a_{30}b_{40}-a_{40}b_{30})=0$.
When $a_{30}b_{30}\neq0$, then there exists a solution $\lambda$ which determines a plane $\pi_1$ on $\mathcal C$. 

(a4)  Generically the singularity of $P_{\pi_1}$ is of Type 6  when the coefficients 
$$
a_{50}b_{30}^2-a_{30}b_{30}b_{50} + (a_{30}b_{40}-a_{40}b_{30})(a_{21}b_{30}-a_{30}b_{21}+b_{40})
$$ 
of $x^5$ in the Taylor expansion of $g$ is not zero and  denoting by $g_{ki}$ the coefficient of $x^{k}y^i$ 
in the Taylor expansion of $g$,
\begin{equation*}
\label{eq:Type7}
(8g_{50}g_{70}-5g_{60}^2)g_{11}+2g_{50}(g_{21}g_{60}-20g_{31}g_{50})g_{11}+35g_{21}^2g_{50}^2\neq0.
\end{equation*}

\qed

\end{proof}

\begin{remark}
Note that, on $\mathcal S_2 \cap \mathcal S_3$ in Theorem \ref{theo:Inf} case {\normalfont (B)}, that is, when $\alpha=\beta=0$, we have $j^2g \equiv 0$. Hence, there exists a line $\mathcal L = \mathcal S_2 \cap \mathcal S_3 \subset \Pi$ such that, for $\pi \in \mathcal L$, the projection $P_\pi$ is $\mathcal{A}$-equivalent to a singularity of type $4_2^\pm$ if $b_{30} \neq 0$ and $b_{21}^2 - 3 b_{12} b_{30} \neq 0$.
\end{remark}

\

\subsection{Case $\pi=T_pM$}

\

We project to the plane $\pi_t$ generated by ${\bf e_3}$ and ${\bf e_4}$. Then 
$P_{\pi}({\bf x})=(x_3,x_4)$
and its restriction to the surface $M$ is the corank 2 map-germ 
\begin{equation}\label{eq:case2}
P_{\pi}(x,y)=(f^1(x,y),f^2(x,y)).
\end{equation}
At the 2-jet level, we have $j^2P_{\pi}(x,y) = (Q_1(x,y),Q_2(x,y))$ and we have the classification in the 2-jet space given by that of pairs of quadratic forms in Table \ref{Gibson}.

This case already been analyzed by Nogueira in \cite{Nogueira} and proves the following.

\begin{theorem}[\cite{Nogueira}]
\leavevmode
\begin{enumerate}
\item[(i)]
At a hyperbolic point of $M$, the projection $P_{\pi}$ is the singularity of type $I^{l,m}_{2,2}$ with $3 \ge l \ge m \ge 1$ (see Table \ref{corank2}). The projection has a singularity of type $I^{1,1}_{2,2}$ at $p$ if and only if the osculating hyperplanes to $M$ at $p$ have exactly $A_2$-contact with $M$ at $p$. If one of these hyperplanes has $A_3$-contact (or worse) with the surface, then the singularity of the projection at $p$ is of type $I^{2,1}_{2,2}$ (or worse).

\item[(ii)]
At an elliptic point, only singularities of type $II^{\,l}_{2,2}$ occur, with $l=1$ or $2$ (see Table \ref{corank2}).

\item[(iii)]
At a parabolic point, the singularities involved are of $\mathcal K$-type $I_{2,m}$ with $m \ge 3$, and these singularities are non-simple.
\end{enumerate}
\end{theorem}

\section{Apparent contours}\label{Apparent contour}

In this section, considering the setting of Section \ref{case1}, we study the apparent contour of the surface $M$ under projection onto the plane $\pi_t$. In this context, we provide a geometric characterization of the surface in terms of the inflection points, vertices, and singularities of the apparent contour when the singular set $\Sigma_\pi$ is a regular curve. To this end, we introduce the necessary concepts from Euclidean differential geometry of surfaces in $\mathbb{R}^4$.

Here we have the induced scalar product of $\mathbb{R}^{4}$ in $T_pM$ and define the second fundamental form of $M$. Given a smooth surface $M\subset\mathbb{R}^{4}$ and $F:U\rightarrow\mathbb{R}^{4}$
a local parametrisation of $M$ with $U\subset\mathbb{R}^{2}$ an open subset, let
${\bf E = \{e_1, e_2, e_3, e_4\}}$ be an orthonormal frame of $\mathbb{R}^{4}$ such that at any $q\in U$,
$ \{{\bf e_{1}}(q),{\bf e_{2}}(q)\}$ is a basis for $T_{p}M$ and $\{{\bf e_{3}}(q),{\bf e_{4}}(q)\}$ is a basis for $N_{p}M$ at $p=F(q)$.

The {\it second fundamental form} of $M$ at a point $p$ is defined by
$II_{p}:T_{p}M\times T_pM^{2}\rightarrow N_{p}M$ given by
$II_p(u,w)=\pi_2(d^2F(u,w))$,
where $\pi_2:T_p\mathbb{R}^{4}\rightarrow N_pM$
is the canonical projection on the normal space. Thus the 
 curvature ellipse can be seen as 
 $\eta(u)=II_p (u,u)$ for $u$ an unit vector in $T_pM$. The {\it second fundamental form of $M^2$ at $p$ along a normal vector field $\nu$} is the bilinear map $II_p^\nu :T_pM^2 \times T_pM^2\to \mathbb{R}$ given by $II_p^{\nu}(v,w)=\langle II_p(v, w),\nu\rangle$.

Taking ${\bf u}=w_{1}{\bf e_{1}}+w_{2}{\bf e_{2}}\in T_pM^{2}$ and $(x,y)$ local coordinates in $\mathbb{R}^2$, 
\[
a=\langle F_{xx},{\bf e_{3}}\rangle,\quad b=\langle F_{xy},{\bf e_{3}}\rangle,\quad c=\langle F_{yy},{\bf e_{3}}\rangle,
\]
\[
l=\langle F_{xx},{\bf e_{4}}\rangle,\quad m=\langle F_{xy},{\bf e_{4}}\rangle,\quad n=\langle F_{yy},{\bf e_{4}}\rangle,
\]
are the {\it coefficients} of the second fundamental form. Then the $II_p$ is a quadratic form which can be written as
\begin{eqnarray*}
II_{p}({\bf u},{\bf u})&=&(aw_{1}^{2}+2bw_{1}w_{2}+cw_{2}^{2}){\bf e_{3}}
+(lw_{1}^{2}+2mw_{1}w_{2}+nw_{2}^{2}){\bf e_{4}}\\
&=& \langle h^1 {\bf u}, {\bf u} \rangle \, {\bf e_{3}} + \langle h^2 {\bf u}, {\bf u} \rangle \, {\bf e_{4}},
\end{eqnarray*}
%
%
%
where
\[
h^1 =
\begin{pmatrix}
a & b \\
b & c
\end{pmatrix},
\quad
h^2 =
\begin{pmatrix}
l & m \\
m & n
\end{pmatrix}.
\]

The map $G^\nu : M \to S^{3}$ given by $G^\nu(p) = \nu(p)$ is
called the {\it Gauss map of $M$ with respect to the normal vector field $\nu$.} The derivative of the Gauss map $G_\nu$ at a point $p$ in $M$ is a linear map
$
(dG^\nu)_p : T_pM \to T_p\mathbb{R}^{4}
= T_pM \oplus N_pM.
$
Let $\pi_1 : T_pM \oplus N_pM \to T_pM$ denote the projection to the first component. The linear map 
\[
W^\nu_p = -\pi_1 \circ (dG^\nu)_p : T_pM \to T_pM.
\]
is called the 
{\it $\nu$-shape operator of $M$ at $p$}. The {\it Lipschitz--Killing curvature of $M$ at $p$ with respect to $\nu$} is defined to be 
$
K^\nu(p) = \det W^\nu_p.
$

Given a vector ${\bf w}=\lambda_1 {\bf e_{3}} + \lambda_2 {\bf e_{4}} \in N_pM$, consider the shape operator
$W^\nu_p$ along any normal vector field with $\nu(p)={ \bf w}$. Then the matrix of $W^\nu_p$ with respect to the basis ${\bf \{e_1, e_2\}}$ is the symmetric matrix
\[
\begin{pmatrix}
\lambda_1 a + \lambda_2 l & \lambda_1 b + \lambda_2 m \\
\lambda_1 b + \lambda_2 m & \lambda_1 c + \lambda_2 n
\end{pmatrix}.
\]
Thus
$$
K^\nu(p) = \det W^\nu_p=(ac-b^2)\lambda_1^2 +(an+cl-2bm)\lambda_1\lambda_2+(ln-m^2)\lambda_2^2$$

Note that the Lipschitz-Killing curvature defines a quadratic form on the normal space. The normal directions for which this form vanishes correspond exactly to the binormal directions of the surface $M$, i.e., $\mathcal B({\bf w})=0$ (see Equation \eqref{binormaleq}).

Consider the case where $\dim(T_pM\cap \pi)=1$. Analogously to the case of surfaces in $\mathbb{R}^3$, we refer to the singular set $\Sigma_\pi$ as the {\it contour generator} of $M$ along $\pi$, and to its image under $P_\pi$ on $\pi_t$ as the {\it apparent contour} $\Delta_{ac}$ of $M$ along $\pi$, that is,
$$
\Delta_{ac}=P_\pi(\Sigma_{\pi}).
$$

With the above notation, and considering Equation~\eqref{Proj1}, the contour generator $\Sigma_\pi$, determined by the equation $g_x(x,y)=0$, can be written as
\begin{eqnarray*}
g_x(x,y)&=&
2\left((a_{02}\alpha^2+a_{11}\alpha+a_{20})\beta-(b_{02}\alpha^2+ b_{11}\alpha+b_{20})\mu\right)x+\\ &&\left((2a_{02}\alpha+a_{11})\beta-(2 b_{02}\alpha+b_{11})\mu\right)y+O_2(x,y)\\
&=& (\langle II_p({\bf u}),{\bf w}\rangle)x+(\langle II_p'({\bf u}),{\bf w}\rangle)y+O_2(x,y),
\end{eqnarray*}
where  $II_p({\bf u}):=II_p({\bf u},{\bf u})$, $II_p'({\bf u})=\frac{dII_p}{d\alpha}({\bf u})$, ${\bf u}=(1,\alpha,0,0)$, and  ${\bf w}=(0,0,\beta,-\mu)$. 
\begin{remark}
Note that $j^1 g_x(x,y)\equiv0$ if and only if 
$
\langle II_p({\bf u}),{\bf w}\rangle = \langle II_p'({\bf u}),{\bf w}\rangle = 0,
$
that is, $II_p({\bf u})$ and $II_p'({\bf u})$ are linearly dependent vectors in $N_pM$. At a non-inflection point, this is equivalent to ${\bf u}$ being an asymptotic direction at $p$ and ${\bf w}$ the associated binormal direction.
\end{remark}


Away from inflection points, if ${\bf u}$ is not asymptotic
($\Omega({\bf u})\neq0$), then $j^1 g_x(x,y)\neq 0$.
By Proposition~\ref{prop:class}, the contour generator
$\Sigma_\pi$ is a regular curve, and two cases arise.
If $\langle II_p({\bf u}),{\bf w}\rangle \neq 0$, then
$P_\pi$ has a fold singularity.
If $\langle II_p({\bf u}),{\bf w}\rangle = 0$ and
$\langle II_p'({\bf u}),{\bf w}\rangle \neq 0$, then
the singularity of $P_\pi$ is of Type 3, 5,$\ldots$, 10.

At inflection points, every tangent direction ${\bf u}$
is asymptotic. Nevertheless, the conditions
$\langle II_p({\bf u}),{\bf w}\rangle \neq 0$ and
$\langle II_p({\bf u}),{\bf w}\rangle = 0$ with
$\langle II_p'({\bf u}),{\bf w}\rangle \neq 0$
still distinguish the cases where the contour generator
$\Sigma_\pi$ is a regular curve, even though ${\bf u}$ is asymptotic. Let us study each case separately.

\subsection{Case 1:} Assuming that 
$\langle II_p({\bf u}),{\bf w}\rangle\neq0$, we parametrize $\Sigma_\pi$ by $t \mapsto (x(t),t)$ with $x(0)=0$, and
$$
x'(0)=-\frac{(2a_{02}\alpha+a_{11})\beta-(2 b_{02}\alpha+b_{11})\mu}{2\left((a_{02}\alpha^2+a_{11}\alpha+a_{20})\beta-(b_{02}\alpha^2+ b_{11}\alpha+b_{20})\mu\right)}=-\frac{\langle II'({\bf u}),{\bf w}\rangle}{\langle II({\bf u}),{\bf w}\rangle }.
$$
Then $P_\pi(x(t),t))=(t,g(x(t),t)))$ is a regular curve such that  $j^2P_\pi(x(t),t))=(t,\displaystyle(\kappa_{ac}/2)t^2)$, where $\kappa_{ac}$ is the curvature of the apparent contour at $t=0$ given by
\begin{equation}\label{apparent}
\kappa_{ac}=\frac{K^\nu}{\langle II({\bf u}),{\bf w}\rangle},
\end{equation}
with $\nu(p)={\bf w}$ and  
$$
K^\nu=(4a_{02}a_{20}-a_{11}^2)\beta^2
-(4(a_{20}b_{02}+a_{02}b_{20})-2a_{11}b_{11})\beta\mu+(4b_{02}b_{20}-b_{11}^2)\mu^2.
$$
Define
$
\sigma_{\mathbf{u}}(\mathbf{w})
=
\langle II_p(\mathbf{u}),\mathbf{w}\rangle.
$
Then
\[
K^\nu=\kappa_{ac}\,\sigma_{\mathbf{u}}(\mathbf{w}).
\]
Motivated by Koenderink's Theorem, we call
\(\sigma_{\mathbf{u}}(\mathbf{w})\)
the \emph{Koenderink factor} associated with the tangent direction
\(\mathbf{u}\) and the normal direction
\(\mathbf{w}\).
The above identity shows that this quantity is precisely the factor relating the Lipschitz--Killing curvature of the surface to the curvature of its apparent contour. Hence, it provides a natural analogue of Koenderink's formula for surfaces in \(\mathbb{R}^4\).


\begin{theorem}\label{Koenderink}
Let $p$ be a point on a surface $M$ in $\mathbb{R}^4$, and let $\pi$ be a $2$-plane such that the projection $P_\pi$ has a fold singularity at $p$. If ${\bf u}\in T_pM\cap\pi$, then the Lipschitz--Killing curvature $K^\nu(p)$ is equal to the product of the curvature $\kappa_{ac}$ of the apparent contour and the Koenderink factor $\sigma_{\bf u}({\bf w})$, namely,
\[
K^\nu(p)=\kappa_{ac}\cdot\sigma_{\bf u}({\bf w}).
\]
\end{theorem}

\begin{remark}
The Koenderink factor depends on the choice of representatives of the tangent and normal directions. However, its vanishing is independent of these choices, and the above identity provides its geometric interpretation as the factor relating the curvature of the apparent contour to the Lipschitz-Killing curvature.
\end{remark}

Recall that if $\kappa$ is the curvature of a plane curve $\gamma(t)$ then an {\it inflection of order} $\ell$ of $\gamma$ is a point where
\[ 
\kappa(t)=\kappa'(t)=\ldots=\kappa^{(\ell-1)}(t)=0 \mbox{ and } \kappa^{(\ell)}(t)\neq 0.
\]
A {\it vertex of order} $\ell$ is a point where $\kappa(t)\neq0$,
\[ 
\kappa'(t)=\ldots=\kappa^{(\ell}(t)=0 \mbox{ and } \kappa^{(\ell+1)}(t)\neq 0.
\]
Inflections and vertices of order one are simply called inflections and vertices. So, as a consequence, we have the following observation.

\begin{corollary}\label{inflection/vertex}
With the conditions of Theorem \ref{Koenderink}, the following holds: 
\begin{itemize}
\item[(i)] $p$ is an elliptic point or an inflection point of $M$ if and only if $P_\pi(p)$ is  not an inflection of the apparent contour along $\pi$. 
\item[(ii)] $p$ is a hyperbolic or parabolic point of $M$ with ${\bf w}$ a binormal direction at $p$ if and only if $P_\pi(p)$ is an inflection of the apparent contour along $\pi$. 
\end{itemize}
\end{corollary}
\begin{proof} The proof follows from Theorem~\ref{theo:notInf}~$(a1)$ and Theorem~\ref{theo:Inf} ((A)(a1) and (B)(a1)).
Since $P_\pi$ is a fold, it follows that $\sigma_{\bf u}({\bf w}) \neq 0$. Thus, statement {\normalfont (i)} follows directly from the fact that $\kappa_{ac} \neq 0$ if and only if $K^\nu(p)=\mathcal{B}({\bf w}) \neq 0$ (that is, when ${\bf w}$ is not a binormal direction). Since elliptic points admit no binormal directions, the result follows.
At inflection points of $M$, to simplify the calculations, we may assume the normal forms given in Table~\ref{Gibson}. In one case,
$
j^2(f^1,f^2)=(x^2+y^2,0),
$
and in the other,
$
j^2(f^1,f^2)=(xy,0).
$
Then we obtain
$
\kappa_{ac} =
\dfrac{2\beta^2}{(\alpha^2+1)\beta}
$
and
$
\kappa_{ac} =
-\dfrac{\beta^2}{2\alpha\beta},
$
respectively. In both cases, $P_\pi$ is a fold if and only if
$\beta\neq0$, equivalently,
$K^\nu=\mathcal{B}({\bf w})\neq0$.
Hence,
$\kappa_{ac}\neq0$, and therefore
$P_\pi(p)$ is not an inflection point of the apparent contour.


\noindent For {\normalfont (ii)}, as $P_\pi$ is a fold, then ${\bf u}\in T_pM\cap\pi$ is not an asymptotic direction. Again, we assume normal forms as in Table \ref{Gibson}. If $p$ is a hyperbolic point, take $j^2(f^1,f^2) = (x^2,y^2)$, and if $p$ is a parabolic point, take $j^2(f^1,f^2) = (xy,y^2)$. In these cases, we obtain
$
\kappa_{ac} = -\dfrac{2\beta\mu}{\beta-\alpha^2\mu }
$
\text{at hyperbolic points,} and
$
\kappa_{ac} = -\dfrac{1}{2}\dfrac{\beta^2}{\alpha\beta-\alpha^2\mu }
$
at parabolic points. Moreover, if ${\bf w}$ is a binormal direction, then $K^\nu = \mathcal{B}({\bf w}) = 0$, which implies that $\beta = 0$ in both cases. Hence, $\kappa_{ac} = 0$ and $P_\pi(p)$ is an inflection point of the apparent contour. Conversely, since $K^\nu = \kappa_{ac}\cdot \sigma_{\bf u}({\bf w})$, if $\kappa_{ac} = 0$, then $K^\nu = \mathcal{B}({\bf w}) = 0$. Therefore, ${\bf w}$ is a binormal direction at $p$, which occurs at hyperbolic or parabolic points.

\qed
\end{proof}

\begin{proposition} Suppose that ${\bf u}\in T_pM\cap\pi$ is not an asymptotic direction
and $p$ is a hyperbolic or parabolic point of $M$ with ${\bf w}$ a binormal direction at $p$.
\begin{itemize}
  \item[(i)] If $p$ is a hyperbolic point, then the inflection becomes of order 2 along the $11_5$-curve on the surface and order three at the singularity Type $18$ of the $P_\pi$ on this curve. 
\item[(ii)] If $p$ is a parabolic point, then the inflection becomes of order 2 at a $P_3(c)$-point on the surface.  
\end{itemize}
\end{proposition}
\begin{proof} As ${\bf u}=(1,\alpha,0,0)\in\pi$ is not asymptotic and $p$ is a hyperbolic (resp. parabolic) with $j^2(f^1,f^2) = (x^2,y^2)$, (resp.   $j^2(f^1,f^2) = (xy,y^2)$), then $\alpha\neq0$. Moreover, assuming that $P_\pi(p)$ is a inflection we have that $\kappa_{ac}=0$. Thus by Theorem \ref{Koenderink} $K^\nu=\mathcal B({\bf w})=0$, where ${\bf w}=(0,0,0.-\mu)$ is a binormal direction.
 Then the apparent contour in both situations can be parametrized by
$$
\left(t,\frac{b_{30}}{\alpha^3}t^3+\frac{\mu(12b_{30}\alpha\lambda+(b_{21}^2-4b_{40})\alpha^2 + 6b_{21}b_{30}\alpha+9b_{30}^2 )}{\alpha^6}t^4+O(5)\right).
$$
The result follows from Proposition \ref{condSingA3} and Proposition \ref{CondP3} using the values of $b_{30}$ and $b_{21}^2-4b_{40}$.
\qed
\end{proof}

\


We now consider the case where $p$ is either an elliptic point or an inflection point, which, by Corollary~\ref{inflection/vertex}~(i), is equivalent to saying that $P_\pi(p)$ is not an inflection point of the apparent contour along $\pi$.
In the elliptic case, ${\bf u}$ is not an asymptotic direction, whereas in the inflection case every direction ${\bf u}\in T_pM$ is asymptotic.
Our goal is to determine when $P_\pi(p)$ is a vertex of the apparent contour.

For the elliptic case, recall that, by Theorem~\ref{theo:notInf}~$(a1)$, the projection $P_\pi$ has a fold singularity whenever $\pi\notin\mathcal S$, where $\mathcal S$ is a surface in the set of planes $\Pi\subset Gr(2,4)$. Moreover, $\Pi$ is parametrized by $(\alpha,\lambda,\mu)\in\mathbb{R}^3$, after setting $\beta=1$.

\begin{theorem}\label{vertex} Let $p$ be an elliptic point satisfying the above conditions. Then we have the following:

\begin{itemize}
\item[(i)] 
For $\pi \in \mathcal S_{v} \setminus \mathcal C_{v}$, where $\mathcal S_v$ is a surface and $\mathcal C_v$ is a curve in $\Pi$, the apparent contour has an ordinary vertex.

\item[(ii)]  For $\pi\in \mathcal C_v$ and away from isolated special planes $\pi_v$ on this curve, the apparent contour has a vertex of order 2. 

\item[(iii)]  For each special plane $\pi_v$, determined by special direction in $T_pM$, the apparent contour has a vertex of order 3. For a generic surface $M$, the vertex becomes of order 4 along a regular curve on the surface $M$ and of order 5 at isolated points on that curve. 

\end{itemize}

\end{theorem}
\begin{proof} For vertices, we use \cite[Theorem 6.3.6]{TariSalaHase}. We compute the 7-jet of $x(t)$ at $t=0$, which determines the 8-jet at $t=0$ of the parametrization $t\mapsto(t,g(x(t),t))$ of the apparent contour. We find
$$
j^8g(x(t),t)=\sum_{i=2}^8{\bf c_i}t^i
$$
where each ${\bf c_i}$ depends on the coefficients $a_{ij}$, $b_{ij}$, $\alpha$, $\mu$, and $\lambda$. At an elliptic point, our calculations simplify considerably if we consider $j^2(f^1,f^2)= (x^2-y^2, xy)$ as in Theorem \ref{theo:notInf}. 

\noindent (i) When ${\bf c_2}\neq0$, i.e., $\kappa_{ac}\neq0$, the apparent contour has an ordinary vertex if and only if ${\bf c_3}=0$ and ${\bf c_4}-({\bf c_2})^3\neq0$. Since the condition ${\bf c_3}=0$ depends on $\alpha$, $\mu$, and $\lambda$, we may isolate $\lambda$ and obtain a surface $\mathcal S_v$ in $\Pi$ parametrized by $(\alpha,\mu)$. Moreover, substituting this expression for $\lambda$ into ${\bf c_4}-({\bf c_2})^3$, we obtain a polynomial of order $40$ in the variables $(\alpha,\mu)$.

\noindent (ii) The point $P_\pi(p)$ is a vertex of order two if and only if
$
{\bf c_3} = {\bf c_4} - ({\bf c_2})^3 = 0$ and
${\bf c_5} \neq 0.$
In this case, the condition ${\bf c}_4 - ({\bf c}_2)^3 = 0$ defines a curve $\mathcal C_v$ in $\Pi$, while ${\bf c_5}$ depends only on the variable $\alpha$.

\noindent (iii) 
Finally, when ${\bf c_3} = {\bf c_4} - ({\bf c_2})^3 = {\bf c_5}= 0$, and ${\bf c_6} - 2({\bf c_2})^5\neq0$, there exist special isolated planes on the curve $\mathcal C_v$, corresponding to distinguished directions in $T_pM$, for which the vertex is of order $3$. Furthermore, the vertex is of order $4$ if and only if ${\bf c_3} = {\bf c_4} - ({\bf c_2})^3 = {\bf c_5}= {\bf c_6} - 2({\bf c_2})^5=0$, and ${\bf c_7} \ne 0$, and it is of order $5$ if and only if ${\bf c_3} = {\bf c_4} - ({\bf c_2})^3 = {\bf c_5}= {\bf c_6} - 2({\bf c_2})^5={\bf c_7}= 0$, and ${\bf c_8} -5 ({\bf c_2})^7\ne0.$ 

The occurrence of vertices of order four (respectively, five) on apparent contours imposes one (respectively, two) independent conditions on the surface. Consequently, they appear generically along a curve (respectively, at isolated points of that curve).\qed
 
\end{proof}

\

For the inflection point case, recall that, by Theorem~\ref{theo:Inf}, the projection $P_\pi$ has a fold singularity whenever $\pi\notin\mathcal S_i$, $i=1,2,3$, where $\mathcal S_i$ are surfaces in the set of planes $\Pi\subset Gr(2,4)$. Moreover, $\Pi$ is parametrized by $(\alpha,\lambda,\beta)\in\mathbb{R}^3$, after setting $\mu=1$.

\begin{theorem} Let $p$ be an inflection point of $M$ satisfying the above conditions. Then we have the following: 
\begin{itemize}
\item[(A)] Suppose $j^2(f_1,f_2)=(x^2+ y^2,0)$ and $\pi\notin \mathcal S_1$.
\begin{itemize}
\item[(i)] 
For $\pi \in \mathcal S^1_{v} \setminus \mathcal C^1_{v}$, where $\mathcal S^1_v$ is a surface and $\mathcal C^1_v$ is a curve in $\Pi$, the apparent contour has an ordinary vertex.

\item[(ii)]  For $\pi\in \mathcal C^1_v$ and away from isolated special planes $\pi^1_v$ on this curve, the apparent contour has a vertex of order 2. 

\item[(iii)]  For each special plane $\pi^1_v$, determined by special direction in $T_pM$, the apparent contour has a vertex of order 3. 
\end{itemize}
\item[(B)] Suppose $j^2(f_1,f_2)=(xy,0)$ and $\pi\notin \mathcal S_2\cup \mathcal S_3$.
\begin{itemize}
\item[(i)] 
For $\pi \in \mathcal S^2_{v} \setminus \mathcal C^2_{v}$, where $\mathcal S^2_v$ is a surface and $\mathcal C^2_v$ is a curve in $\Pi$, the apparent contour has an ordinary vertex.

\item[(ii)]  For $\pi\in \mathcal C^2_v$ and away from isolated special planes $\pi^2_v$ on this curve, the apparent contour has a vertex of order 2. 

\item[(iii)]  For each special plane $\pi^2_v$, determined by special direction in $T_pM$, the apparent contour has a vertex of order 3. 
\end{itemize}
\end{itemize}
\end{theorem}
\begin{proof} 
The proof is analogous to that of Theorem~\ref{vertex} and is therefore omitted.
\end{proof}

\begin{remark}
The study of the vertices of the apparent contour can be made more robust by analyzing the contact between the surface and the cylinders. This approach provides a deeper understanding of the geometry of higher-order vertices, and will be explored in future work.
\end{remark}

\subsection{Case 2:} Assuming $\langle II_p({\bf u}),{\bf w}\rangle = 0$ and $\langle II_p'({\bf u}),{\bf w}\rangle \neq 0$. Then
$
(a_{02}\alpha^2+a_{11}\alpha+a_{20})\beta-(b_{02}\alpha^2+ b_{11}\alpha+b_{20})\mu=0$
and $\left((2a_{02}\alpha+a_{11})\beta-(2 b_{02}\alpha+b_{11})\mu\right)\ne0$.
If we denote by $g_{ij}$ the coefficient of $x^{i}y^i$ in $g$, then parametrize $\Sigma_\pi$ by $t \mapsto (t,y(t))$ with 
$$
y(t)={\bf d_2}t^2+{\bf d_3}t^3+{\bf d_4}t^4+{\bf d_5}t^5+O(6)
$$
where
$$
\left\{
\begin{array}{l}
{\bf d_2}=-\frac{3g_{30}}{g_{11}};\\ 
{\bf d_3}=-\frac{4g_{11}g_{40}-6g_{21}g_{30})}{g_{11}^2};\\ 
{\bf d_4}=-\frac{5g_{11}^2g_{50}-(8g_{21}g_{40}+9g_{30}g_{31})g_{11}+3g_{30}(g_{12}g_{30}+4g_{12}^2)}{g_{11}^3}.
\end{array}
\right.
 $$

Then the apparent contour $\Delta$ is parametrized by $P_\pi(t,y(t))=(y(t),g(t,y(t)))$ such that 
$$
g(t,y(t))={\bf h_3}t^3+O(4),
$$
where 
$$
\left\{
\begin{array}{l}
{\bf h_3}=-{2g_{30}};\\ 
{\bf h_4}=\frac{9g_{02}g_{30}^2-3g_{11}^2g_{40}+g_{11}g_{21}g_{30}}{g_{11}^2};\\ 
\end{array}
\right.
 $$

A map-germ $\gamma : (\mathbb{R}, 0) \to (\mathbb{R}^2, 0)$ is called a \emph{$(k,\ell)$-cusp singularity} if it is $\mathcal{A}$-equivalent to  $t\mapsto(t^k, t^\ell)$. Recognition criteria for these singularities are known, see for example \cite{BruceGaffiney}. 
In particular, $\gamma$ has a $(2,3)$-cusp at the origin if and only if $\gamma'(0)=0$ and
$
\det\big[\gamma''(0),\gamma'''(0)\big]\neq 0.
$
Similarly, $\gamma$ has a $(3,4)$-cusp at the origin if and only if $\gamma'(0)=\gamma''(0)=0$ and
$
\det\big[\gamma'''(0),\gamma^{(4)}(0)\big]\neq 0.
$
Finally, when
$
\gamma'(0)=\gamma''(0)=\gamma'''(0)=0,
$
we say that $\gamma$ has a \emph{$(4,5)$-cusp or worse}.

\begin{theorem}\label{cups-apparent}
Consider $p\in M$ such that the singularity of $P_\pi(p)$ is of Type 3 or worse ($P_\pi(p)\sim_{\mathcal A^{(2)}}(y,xy)$). 
\begin{itemize}
    \item[(i)] If $P_\pi(p)$ is of Type 3, then the apparent contour has a  $(2,3)$-cusp singularity.
    
    \item[(ii)] If $P_\pi(p)$ is of Type 5, then the apparent contour has a singularity $(3,4)$-cusp singularity.

    \item[(iii)] If $P_\pi(p)$ is of Type 6 or worse, then the apparent contour has a $(4,5)$-cusp singularity or worse.
\end{itemize}
\end{theorem}

\begin{proof}
The proof is divided into two cases, according to whether $p$ is a non-inflection point or an inflection point of $M$.
In the first case, we are under the assumptions of Theorem \ref{theo:notInf} $(a2)$, so the point $p$ may be elliptic, hyperbolic, or parabolic. Accordingly, we may consider
\[
(Q_1, Q_2) = (x^2 - y^2, xy), \quad (x^2, y^2), \quad \text{and} \quad (xy, y^2),
\]
in each of these cases, respectively, as this choice significantly simplifies the computations. As in Theorem \ref{theo:notInf}, we treat the case where $p$ is an elliptic point. The remaining cases follow analogously and are therefore omitted. In such a case, 
the apparent contour $\Delta_{ac}$ is parametrized by 
$$
({\bf d_2}t^2+O(3),{\bf h_3}t^3+O(4))
$$
where ${\bf d}_2=\frac{3\Lambda}{\alpha^2+1}$ and ${\bf h}_3=\frac{-2\Lambda}{\alpha}$, and $\Lambda$ is given by Equation~\eqref{cuspcondition} in the proof of Theorem~\ref{theo:notInf}.

\noindent (i) If $P_{\pi}(p)$ has a singularity of Type $3$, that is, if $\Lambda \neq 0$, then, by the recognition criteria, the apparent contour has a $(2,3)$-cusp.

\noindent (ii) When $\Lambda=0$, the $P_{\pi}(p)$ has a singularity of Type $5$ if $g_{40}\neq0$. In this situation,
$\Delta$ is parametrized by 
$$
({\bf d_3}t^3+O(4),{\bf h_4}t^4+O(5))
$$
where  ${\bf d_3}=-\frac{4g_{40}}{\alpha^2+1}$ and ${\bf h_4}=-3g_{40}$, i.e., it has a $(3,4)$-cusp.

\noindent (iii) Finally, $P_{\pi}(p)$ has a singularity of Type $6$ or worse, if $\Lambda=g_{40}=0$. So 
$\Delta$ is given by  
$$
({\bf d_4}t^4+O(5),{\bf h_5}t^5+O(6))
$$
where  ${\bf d_4}=-\frac{5g_{5
0}}{\alpha^2+1}$ and ${\bf h_5}=g_{50}$, i.e., it has a $(4,5)$-cusp or worse. 

Finally, for the case where $p$ is an inflection point, we assume the conditions of Theorem~\ref{theo:Inf} $(B)$ with
$
(Q_1, Q_2) = (xy, 0).
$
The result then follows analogously to the first case.

\qed
\end{proof}


\begin{remark}
A complete classification of the $\mathcal A$-simple singularities of germs $(\mathbb R,0)\to(\mathbb R^2,0)$ can be seen in \cite{BruceGaffiney}. Moreover, item $(iii)$ of Theorem~\ref{cups-apparent} can be refined further. Indeed, if $P_\pi(p)$ is a singularity of Type $6$, $7$, $8$, $9$, or $10$, then the corresponding apparent contour is $\mathcal A$-equivalent to
\[
(t^4,t^5+t^7), \qquad
(t^4,t^5), \qquad
(t^4,t^6+t^7), \qquad
(t^4,t^6+t^9), \qquad
(t^4,t^7),
\]
respectively. The conditions characterizing each of these cases are precisely those obtained in the proof of Theorem~\ref{theo:notInf} $(a_4)$.
\end{remark}

\vspace{0.5cm} 

\hspace{-0,4cm}{\bf Acknowledgement:}
The authors would like to thank Farid Tari for comments. 

\end{document}